\documentclass[a4paper,12pt]{article}
\usepackage{latexsym}
\usepackage{amsmath}
\usepackage{amssymb}
\usepackage{enumerate}

\usepackage{theorem}
\newtheorem{theorem}{Theorem}[section]
\newtheorem{proposition}[theorem]{Proposition}
\newtheorem{lemma}[theorem]{Lemma}
\newtheorem{corollary}[theorem]{Corollary}

\theorembodyfont{\rmfamily}
\newtheorem{proof}{\textmd{\textit{Proof.}}}

\newtheorem{remark}[theorem]{Remark}
\newtheorem{example}[theorem]{Example}
\newtheorem{definition}[theorem]{Definition}

\newtheorem{acknowledgement}{\textmd{\textit{Acknowledgements.}}}

\makeatletter

\@addtoreset{equation}{section}
\makeatother

\newcommand{\qedd}{\hfill \Box}

\newcommand{\R}{\ensuremath{\mathbb{R}}}
\newcommand{\C}{\ensuremath{\mathbb{C}}}
\newcommand{\Sph}{\ensuremath{\mathbb{S}}}
\newcommand{\E}{\ensuremath{\mathbb{E}}}

\title{Kropina metrics and Zermelo navigation on Riemannian manifolds
\footnote{
Mathematics Subject Classification (2010)\,:\,53C60, 53C22.}
\footnote{
Keywords: Kropina spaces, flag curvature, Killing vector fields, Riemannian isometries.}
}
\author{Ryozo YOSHIKAWA $\cdot$ Sorin V. SABAU}
\date{}
\begin{document}


\maketitle
\begin{abstract}
The present paper studies globally defined Kropina metrics as solutions of the Zermelo's navigation problem. Moreover, we characterize the Kropina metrics of constant flag curvature showing that up to local isometry, there are only two model spaces of them: the Euclidean space and the odd-dimensional spheres.

\end{abstract}

\section{Introduction.}\label{intro}

Finsler metrics were introduced in order to generalize the Riemannian ones in the sense that the metric should depend not only on the point, but also on the direction. This generalization leads to quite complex computations if one wants to find conditions for a Finsler manifold to be of constant flag curvature making a classification of Finslerian spaces of constant flag curvature extremely difficult.

One of the few complete classifications of Finsler manifolds of constant flag curvature was done for a particular class of metrics, the so-called {\it Randers metrics} $F=\alpha+\beta$ (see \cite{BCS}). They belong to a large class of Finsler metrics, the {\it ($\alpha$, $\beta$)-metrics}, where $\alpha=\sqrt{a_{ij}(x)y^iy^j}$ is a Riemannian metric and $\beta=b_i(x)y^i$  a linear 1-form on $TM$. Here we denote by $M$ an $n(\ge 2)$-dimensional differential manifold and by $TM$ its tangent bundle with local coordinates $x$ and $(x, y)$, respectively.

Another remarkable class of Finsler manifolds with ($\alpha$, $\beta$)-metrics are the so-called {\it Kropina metrics} $F=\alpha^2/\beta$ introduced in \cite{K} (one can see in \cite{AIM} an interesting discussion on the role of this metric in thermodynamics). 

Being computationally friendly, Randers spaces were recently studied by many researchers. We mention 
Zhongmin Shen who regarded Randers spaces from a new perspective, identifying these metrics with the solution of the Zermelo's navigation problem on some Riemannian space in \cite{Sh1} and characterized a Randers metric  by a new Riemannian metric $h$ and a vector field $W$ on $(M, h)$ with $|W|<1$. Moreover, in 2004, David Bao, Colleen Robles and Zhongmin Shen have obtained the necessary and sufficient conditions for a Randers space to be of constant flag curvature and completed the classification of strongly convex Randers metrics of constant flag curvature (\cite{BRS}).
Finally, in 2005,  Colleen Robles investigated the geodesics of a Randers space of constant flag curvature (\cite{R}) obtaining the characterization of long geodesics, conjugate points and cut locus of such a metric.

On the other hand, the results on Kropina metrics are quite few. In 1978, Choko Shibata studied some basic local geometrical properties of Kropina spaces (\cite{Shi}).
In 1991, Makoto Matsumoto obtained a set of necessary and sufficient conditions for a Kropina space to be of constant (flag) curvature (\cite{M2}). Based on these results, the first author together with Katsumi Okubo characterized a Kropina metric $F=\alpha^2/\beta$ by means of a new Riemannian metric $h$ and a  unit length vector field $W$ (\cite{YO1}) obtaining a minimal set of necessary and sufficient local conditions for a Kropina space to be of constant curvature (\cite{YO1} and \cite{YO2}). 

Since topology is precisely the mathematical field that allows the passage from local to global, in the present paper we will study the existence of different types of Kropina structures taking into account the topological restrictions involved.

Especially, we are concerned with the existence of such structures globally defined putting in evidence the characteristics of Kropina metrics, namely the fact that unlikely Randers metrics, they cannot be defined on the entire tangent bundle, but only on a conic domain of it. We have found the concept of conic Finsler metrics introduced and studied by M. A. Javaloyes and M. S\'anchez (\cite{JS}) extremely useful. 

Moreover, we will classify globally defined Kropina metrics of constant flag curvature pointing out the differences with the Randers spaces of constant flag curvature.

Here is the content of our paper. 

In Section \ref{section 1} we characterize a class of Kropina metrics, the U-Kropina metrics defined as Kropina metrics  with a unit vector field, as a solution of the Zermelo's navigation problem on Riemannian manifolds, other than those found and studied in \cite{BRS}. We were led in this way to the description of these metrics as conic Finsler metrics (Section \ref{section 2}), a relatively new notion introduced in \cite{JS}. In Section \ref{section 3} we apply the results in \cite{YO2} obtaining a characterization theorem for U-Kropina metrics with constant flag curvature (we call these metrics CC-Kropina metrics). 

However, since our Kropina metrics are defined globally on the manifold $M$, strong topological restrictions to the existence of such structures appear. We study in Sections \ref{section 4} and \ref{section 5} the existence of three globally defined Kropina structures: the U-Kropina, UK-Kropina and CC-Kropina metrics, respectively (see Definition \ref{definition 5.1} for the precise definitions). In particular, we show that CC-Kropina metrics can be constructed only on odd dimensional spheres or (locally) Euclidean spaces (see Theorem \ref{Theorem 4.20}). Moreover, such Kropina metrics are unique up to isometry (Section \ref{section 6}).

Finally, we study the projective flatness of U-Kropina metrics (Section \ref{section 7}) showing that a Beltrami type theorem holds only for the flat case. 

The paper is illustrated with several examples.

Other topics as geodesics behavior, conjugate points, cut points, etc. will be studied in a forthcoming paper. 

\acknowledgement{ We thank to Professors K. Okubo and H. Shimada for encouraging us in the research of this topic.}

\section{Another solution to the Zermelo's navigation problem.}\label{section 1}

In 1931, E. Zermelo studied the following problem (see \cite{C}):

{\it Suppose a ship sails the sea and a wind comes up. How must the captain steer the ship in order to reach a given destination in the shortest time?
}

The problem was solved by Zermelo himself for the Euclidean flat plane and by D. Bao, C. Robles and Z. Shen (\cite{BRS}) in the case when the sea is a Riemannian manifold $(M,h)$ under the assumption that the wind $W$ is a time-independent mild breeze, i.e. $h(W,W)<1$. In the case when $W$ is a time-independent wind, they have found out that the path minimizing travel-time are exactly the geodesics of a Randers metric
\begin{equation*}
F(x,y)=\alpha(x,y)+\beta(x,y)=\frac{\sqrt{\lambda\cdot |y|^2+W_0^2}}{\lambda}-\frac{W_0}{\lambda},
\end{equation*}
where $W=W^i\frac{\partial}{\partial x^i}$ is the wind velocity, $|y|^2=h(y,y)$, $\lambda=1-|W|^2$ and
$W_0=h(W,y)$. 

The Randers metric $F$ is said {\it to solve the Zermelo's navigation problem} in the case of a mild breeze. The condition $h(W,W)<1$ ensures that $F$ is a positive-definite Finsler metric (see \cite{BRS}, \cite{BCS}).

We are going to show that Zermelo's navigation problem has another solution in the case when the wind becomes stronger. Imagine a ship with an engine of 20 knots top speed and a wind of about the same strength. By normalizing, we can consider an open sea represented by a Riemannian manifold $(M,h)$ and a wind $W=W^i\frac{\partial}{\partial x^i}$ such that $h(W,W)=1$. 

If we denote by $u$ the velocity of the ship in the absence of the wind, then in the windy conditions the ship velocity will be given by the composed vector $v=u+W$. Considering the ship sailing with Riemannian unit speed, i.e. $h(u,u)=1$, then our setting is a special case of Zermelo's navigation problem. Before going any further let us remark some basic facts:
\begin{enumerate}
\item Obviously, in windy conditions, the Riemannian metric $h$ no longer gives the travel time along vectors, but a new function $F$ should be introduced on $TM$.
\item Since the ship velocity $|u|$ and the wind strength $|W|$ are equal, clearly, unlike the Randers case,  the ship cannot sail anymore against the wind. In other words, there is a preferential direction, namely $u=-W$, where the resultant vector $v$ vanishes. The implications of this fact are described in the following sections of the paper.
\item Geometrically, one can easily see that in each tangent space $T_xM$, the unit sphere $S_x$ of the new metric $F$ is the $W$-translate of the Riemannian $h$-unit sphere $I_x$. Again, unlike the Randers case, $S_x$ passes through the origin of $T_xM$ and therefore the new metric $F$ cannot be a Finsler metric in the classsic sense. 
\end{enumerate}

We are going to compute this new function $F$ as follows. Start with a Riemannian manifold $(M,h)$ and a vector field $W=W^i\frac{\partial}{\partial x^i}$ on $M$ of $h$-unit length. Then the metric $F$ we are interested in is given by the solution of the equation
\begin{equation}\label{eq for F}
|\frac{y}{F(x,y)}-W|=1.
\end{equation}

From the definition of the inner product it follows
\begin{equation*}
F(x,y)=\frac{|y|^2}{2h(y,W)},
\end{equation*}
provided $W\neq 0$, condition guaranteed by the initial choice of $W$ such that $|W|=1$ everywhere, and $y\neq 0$.

By putting
\begin{equation}\label{a_ij, b_i def}
a_{ij}(x):=e^{-k(x)}h_{ij},\qquad b_i(x):=2e^{-k(x)}W_i,
\end{equation}
we obtain the metric
\begin{equation}\label{Kropina metric def}
F(x,y)=\frac{\alpha^2(x,y)}{\beta(x,y)},
\end{equation}
where $k(x)$ is some smooth function on $M$, $\alpha(x,y)=\sqrt{{a}_{ij}(x)y^iy^j}$ and 
$\beta=b_i(x)y^i$. One can remark that with these notations $b^2:=a^{ij}b_ib_j=4e^{-k(x)}$.

The metric $F$ in \eqref{Kropina metric def} is known as the {\it Kropina metric} and its geometry will be studied in the following sections. 

\begin{remark}
\begin{enumerate}
\item
Let us reflect for a moment at the formulas \eqref{a_ij, b_i def}. From \eqref{eq for F} one can see that $a_{ij}$ and $b_i$ could be introduced without the conformal factor $e^{-k(x)}$. However, in this case $b^2$ would be constant and therefore the Kropina metrics we obtain are subject to this constraint. 
\item It is clear that $a_{ij}(x)$ is a Riemannian metric on $M$.
\end{enumerate}
\end{remark}

Conversely, let us start with a Riemannian structure $(M,a)$ and a 1-form $\beta=b_i(x)y^i$ of Riemannian length $b^2$ not necessarily constant. Then, if we put
\begin{equation}\label{2.5}
h_{ij}:=e^{k(x)}a_{ij},\qquad W_i(x):=\frac{1}{2}e^{k(x)}b_i,
\end{equation}
where 
\begin{equation}\label{2.5'}
k(x)=\log\frac{4}{b^2(x)},
\end{equation} 
we obtain the initial Zermelo's navigation problem in terms of $h$ and $W$ whose solution is precisely the Kropina metric \eqref{Kropina metric def}.

\begin{remark}
\begin{enumerate}
\item For a given Riemannian structure $(M,a)$ and a 1-form $\beta=b_i(x)y^i$ one can see that the norm  \eqref{Kropina metric def} is not defined on all $TM$, but only on a domain $A=\{(x,y)\in TM:\beta> 0\}$.
\item The description given here as the solution of the Zermelo's navigation problem is correct only in the case $|W|=1$, i.e. on $M$ it must exist a vector field nowhere vanishing. This requirement imposes immediately topological restriction on the manifolds admitting well defined Kropina metrics. 
\end{enumerate}
\end{remark}

The remark above leads to the following definition.

\begin{definition}
Let $(M, h)$ be an $n$-dimensional Riemannian space, $W$ a unit vector field globally defined on $M$, and let us consider the $a_{ij}$ and $b_i$ given in \eqref{a_ij, b_i def}. We denote by $F$ the Kropina metric obtained using these $\alpha$ and $\beta$ as explained above. 

Then $F$ will be called {\it Kropina metric with unit vector field}, or {\it U-Kropina metric}.
\end{definition}

Summarizing, we obtain
\begin{proposition}
A metric $F$ is of U-Kropina type if and only if it is solution of the Zermelo's navigation problem on the Riemannian manifold $(M,h)$ under the influence of an $h$-unit wind $W$.
\end{proposition}

This is not a classical Finsler metric, but a conic one (see next section for details). 
\section{Conic Finsler metrics.}\label{section 2}
This is an introductory chapter that follows closely (\cite{JS}).
\subsection{Minkowski conic norms.}\label{subsection 2.1}

Let $V$ be a real vector space and $A$ be a {\it conic domain} of $V$, i.e. $A$ is an open, non-empty subset of $V$ such that if $v\in A$, then $\lambda v \in A$ for all $\lambda >0$.
In particular, it is worth mentioning that the origin of $V$ does not belong to $A$, except the case $A=V$.

One can now define Minkowski norms on such a conic domain as follows.

\begin{definition}(\cite{JS})\label{Minkowski conic norm}
Let $V$ be a real vector space and $A$ be a conic domain of $V$. A {\it Minkowski conic norm} $|| \cdot||$ on $V$ is a map
\[ || \cdot|| : A \longrightarrow (0, \infty),\]
which satisfies the conditions:
\begin{description}
\item{(i)} {\bf strictly positivity} : $||v||> 0$ for any $v\in A$,
\item{(ii)} {\bf positively homogeneity} : $||\lambda v||=\lambda ||v||$ for all $\lambda > 0$,
\item{(iii)} {\bf positive definite Hessian} :\\
       (c1)  $||\cdot||$ is smooth on $A$, so that the {\it fundamental tensor field} $g$ of $||\cdot||$ on $A$ can be defined as the Hessian of $\frac{1}{2}||\cdot||^2$,\\
   (c2) $g$ is pointwise positive definite on $A$ .
\end{description}
\end{definition}

It follows that for a given Minkowski conic norm $||\cdot||$ and $v \in A$, the fundamental tensor $g_v$ is given as :
\[   g_v(u, w):=\frac{\partial^2}{\partial t \partial s}G(v+t u+s w)\bigg|_{t=s=0},  \]
where  $v+t u+s w\in  A$ and $G=|| \cdot||^2/2$.
Moreover, $v \mapsto g_v$ is positively homogeneous of degree 0 (that is, $g_{\lambda v}=g_v$ for $\lambda >0$) and it satisfies 
\[ g_v(v, v)=||v||^2, \hspace{0.2in} g_v(v, w)=\frac{\partial}{\partial s}G(v+s w)\bigg|_{s=0},\]
where $v+s w \in A$.

\begin{proposition}{\rm (\cite{JS})}\label{Proposition 1.2}
Let $|| \cdot || \longrightarrow (0, \infty)$ be a Minkowski conic norm. Then, the unit sphere
\[   S:=\{v\in A  :  ||v||=1\}     \]
is a hypersurface embedded in $A$ as a closed subset, and the position vector at each point is transverse to $S$.
\end{proposition}

The unit sphere  $S$ defined in Proposition \ref{Proposition 1.2} will be called the {\it indicatrix} of a Minkowski conic norm $|| \cdot||$.

Similarly to the classical Finslerian case, the unit balls of Minkowski conic norms have special properties. Indeed, one has

\begin{proposition}{\rm(\cite{JS})}\label{unit ball prop}
Let $|| \cdot|| : A \longrightarrow (0, \infty)$ be a Minkowski conic norm. Then
\begin{description}
\item (1) The ball defined by
  \[B:=\{ v\in A :  ||v||\le 1\}          \]
   is a closed subset of $A$ which intersects all the directions $D_v:=\{\lambda v | \lambda >0\}$, where $v\in A$.
\item(2) $B$ is starshaped from the origin, i.e., $v\in B$ implies $\lambda v \in B$ for all $\lambda \in (0, 1)$.
\item(3) The boundary $S$ of $B$ in $A$ is a smooth hypersurface and a closed subset of $A$ such that the position vector at each $v\in S$ is transversal.
\item(4) For each $v\in B$ there exists a (necessarily unique) $\lambda >0$ such that $v/\lambda\in S$.
\item(5) $S$ is homeomorphic to an open subset of the usual sphere.
\end{description}
\end{proposition}

Conversely, if $A$ is a conic domain in $V$, and $B\subset A$ is a subset that satisfies the properties in Proposition \ref{unit ball prop}, then the map
\begin{equation*}
||\cdot||_B:A\to \R,\qquad v\mapsto \inf\{\alpha\geq 0:\frac{v}{\alpha}\in B\}
\end{equation*}
is a Minkowski conic norm and its closed unit ball coincides with $B$.


\subsection{Conic Finsler metrics.}\label{subsection 2.2}

Minkowski conic norms allow to define  conic Finsler metrics.

\begin{definition}(\cite{JS})
Let $M$ be an $n$-dimensional differentiable manifold and $A\subset TM$ be an open subset of the tangent bundle $TM$ such that $\pi(A)=M$, where $\pi : TM \longrightarrow M$ is the natural projection, and let $F : A \longrightarrow [0, \infty)$ be a continuous function. 
Assume that $(A, F)$ satisfies:
\begin{description}
  \item{(i)} $A$ is {\it conic} in $TM$, i.e.  for each $x\in M$, $A_x:=A\cap T_xM$ is a conic domain in $T_xM$.   
 \item{(ii)} $F$ is smooth on $A$.
\end{description}


One says that  $(A, F)$, or simply $F$, is a {\it conic Finsler metric} if each $F_x$ is a Minkowski conic norm, i.e. the fundamental tensor $g$ on $A$  induced by all the fundamental tensor field $g_x$  on $A_x$ is positive definite for each $x\in M$. A Finsler space with a conic Finsler metric is called a {\it conic Finsler space}.
\end{definition}

Following \cite{JS}, let us observe that the fundamental tensor $g$ can be thought as a section of a fiber bundle over $A$. To be more precise, denote the restriction to $A$ of the natural projection 
$\pi : TM \longrightarrow M$ 
by
$\pi_A : A \longrightarrow M$
and define 
$\pi^* : \pi_A^*(TM) \longrightarrow A$
to be the fiber bundle obtained as the pulled-back bundle of $\pi : TM \longrightarrow M$ thorough $\pi : A \longrightarrow M$, namely we have
\begin{eqnarray*}
 \pi_A^*(TM)    \hspace{0.2in}          \longrightarrow \hspace{0.2in}                TM           \nonumber\\
          \pi^*     \downarrow     \hspace{1in}     \downarrow \pi   \label{diagram} \\
A \hspace{0.4in}  \longrightarrow  \hspace{0.2in}       M.            \nonumber
\end{eqnarray*}

Then $g$ is a smooth symmetric section of the fiber bundle $\pi_A^*(TM^*)\otimes \pi_A^*(TM^*)$ over $A$, where $\pi_A^*(TM^*)$ is the dual of $\pi_A^*(TM)$. Let us remark that if we fix a vector $v\in A$, then $g_v$ is a symmetric bilinear form on $T_{\pi{(v)}}M$.

Geometrically, over each $(x,y)\in A$ we erect a copy of $T_xM$ endowed with the inner product $g_{ij}(x,y)dy^i\otimes dy^j$, $(x,y)\in A$. The resulting vector bundle of fiber dimension $n$ is the pullback bundle over $A\subset \widetilde TM$ denoted as $(\pi_A^*(TM),  \pi_A^*, A)$. 

\begin{remark}
We remark that, due to its $0$-homogeneity, $g_{ij}(x,y)$ is constant along each ray contained in $A_x$ that emanates from the origin of $T_xM$. Therefore, at each $x\in M$, one can consider the ray
\begin{equation*}
(x,[y]):=\{(x,\lambda y)\in A: \lambda>0\},
\end{equation*}
and denote the set of such elements by $S_AM$. 

Over each ray $(x,[y])\in S_AM$ one can now erect a copy of $T_xM$ and introduce the inner product $g_{ij}(x,y)dy^i\otimes dy^j$ there. We obtain in this way a $(2n-1)$ dimensional manifold $S_AM$ and a vector bundle of fiber dimension $n$ constructed over $S_AM$ with the fiber metric $g$ given above. However, due to computational issues, it is more convenient to work with the affine coordinates of $A\subset \widetilde{TM}$.
\end{remark}

As in the classical case, there is a global section of $\pi_A^*(TM)$ given by
\begin{equation*}
l_A(x,y):=\frac{y^i}{F(x,y)}\frac{\partial}{\partial x^i},\quad (x,y)\in A
\end{equation*}
called the {\it canonical section}. 

Likely, the dual vector bundle $\pi_A^*(TM^*)$ also has a distinguished global section 
\begin{equation*}
\omega:=\frac{\partial F}{\partial y^i}dx^i,
\end{equation*}
the {\it Hilbert form} of $F$. 

On $A\subset \widetilde{TM}$ we can introduce the Christoffel symbols of the second kind
\begin{equation*}
\gamma_{jk}^i:=\frac{1}{2}g^{is}
\bigl( \frac{\partial g_{sj}}{\partial x^k}- \frac{\partial g_{jk}}{\partial x^s}+
\frac{\partial g_{ks}}{\partial x^j} \bigr)
\end{equation*}
and the canonical spray coefficients
\begin{equation*}
G^i(x,y):=\frac{1}{2}\gamma_{jk}^i(x,y)y^jy^k,\quad (x,y)\in A.
\end{equation*}

Then the nonlinear connection has the coefficients
\begin{equation*}
N_j^i(x,y):=\frac{\partial G^i(x,y)}{\partial y^j},\quad (x,y)\in A
\end{equation*}
and using it one can change the natural basis $\{\frac{\partial}{\partial x^i}, \frac{\partial }{\partial y^i}\}$  of $T_uA$ to the adapted basis $\{\frac{\delta}{\delta x^i}:=\frac{\partial}{\partial x^i}-{N_i}^j\frac{\partial} {\partial y^j}, \frac{\partial }{\partial y^i}\}$.

The Berwald connection coefficients are given by 
\begin{equation*}
\Gamma_{jk}^i(x,y):=\frac{\partial N_j^i}{\partial y^k}=\frac{\partial^2 G^i}{\partial y^j\partial y^k}
\end{equation*}
and using it one has the Berwald $h$-curvature tensor
\begin{equation*}
R_{j\ kl}^{\ i}:=\frac{\delta \Gamma^i_{jk}}{\delta x^l}+\Gamma_{jk}^r\Gamma_{rl}^i-\textrm{terms with $k$, $l$ interchanged}.
\end{equation*}


\begin{definition} 
Let $(M, F)$ be an $n(\ge 2)$-dimensional conic Finsler space, where $F : A\longrightarrow (0, \infty)$ is a conic Finsler metric. 

The flag curvature $K(x, y, X)$ of a conic Finsler space is defined by
\begin{equation*}\label{flag curvature}
	K(x, y, X):=\frac{R_{hijk}y^hX^iy^jX^k}{(g_{hj}g_{ik}-g_{hk}g_{ij})y^hX^iy^jX^k},
\end{equation*}
where
\[\mathcal{Y}=y^i\frac{\partial }{\partial x^i}, \hspace{0.1in} \mathcal{X}=X^i(x, y)\frac{\partial }{\partial x^i}\in A. \]

If $K(x, y, X)$ is independent of $X$, the conic Finsler space is said to be of {\it scalar flag  curvature}. 
Furthermore,  if $K(x, y, X)$ is constant for any $x$, $y$ and $X$, the conic Finsler space is called to be of {\it constant flag curvature}.
\end{definition}

If a conic Finsler space $(M, F)$, where $F : A \longrightarrow (0, \infty)$, is of scalar flag curvature, then, similarly with the classical Finsler geometry, we get

\begin{proposition}\label{Proposition 3.1}
  The necessary and sufficient condition for a conic Finsler space $(M, F)$,  with the Finsler conic metric 
  $F : A \longrightarrow (0, \infty)$, to be of scalar flag curvature $K$ is that the equation
 \begin{equation}\label{scalar curvature}
	 {{R_0}^i}_{0l}=KF^2 {h^i}_l
\end{equation}
holds, 
where 
\[{h^i}_l={\delta^i}_l-l^il_l, \hspace{0.1in} l^i=\frac{y^i}{F}, \hspace{0.1in} l_i=\frac{\partial F}{\partial y^i}.  \] 
\end{proposition}

The geometrical quantities in the equation (\ref{scalar curvature}) depend on $x$ and $y$. 
We must notice that $y$ is restricted by $y\in A_x$ for any $x\in M$.
This is the main difference with the classical theory of Finsler spaces.


\section{Kropina metrics of constant flag curvature.}\label{section 3}

\subsection{Kropina metrics as conic Finsler metrics.}

Let $M$ be an $n(\ge 2)$-dimensional differential  manifold endowed with a Riemannian metric $a=(a_{ij}(x))$, and recall that a Kropina space $(M,F)$ is the Finsler space with the norm $F(x, y):=\alpha^2/\beta$, where $\alpha^2=a_{ij}(x)y^iy^j$ and $\beta=b_i(x)y^i$.

For a Kropina metric $F$, at each $x\in M$ we
define
\begin{equation}\label{domain-b}
     A_x:=\{ y=y^{i}\frac{\partial}{\partial y^{i}}\in T_x M  | b_i(x)y^i>0\},
\end{equation}
which is a conic domain of $T_x M$ whose boundary is the hyperplane $\{ y=y^{i}\frac{\partial}{\partial y^{i}}\in T_x M  | b_i(x)y^i=0\}$.

\begin{proposition}
 The Kropina metric $F:A\to (0,\infty)$ is a conic Finsler metric, where $A_x\subset T_xM$ is the conic  domain \eqref{domain-b}, for any $x\in M$.
\end{proposition}
\begin{proof}
A straightforward computation shows that the Hessian matrix of the Kropina metric $F=\alpha^2/\beta$ can be written as
\begin{equation*}
g_{ij}(x,y)=\frac{2\alpha^2}{\beta^2}a_{ij}+\frac{3\alpha^4}{\beta^4}b_ib_j-\frac{4\alpha^2}{\beta^3}(a_{0i}b_j+a_{0j}b_i)+\frac{4}{\beta^2}a_{0i}a_{0j},
\end{equation*}
where $a_{0i}=a_{ji}y^j$. 

It follows	
\begin{equation*}
g_{ij}(x,y)v^iv^j=\frac{2}{\beta^2}\alpha^2a_{ij}v^iv^j+\frac{3\alpha^4}{\beta^4}(b_iv^i)^2-\frac{8\alpha^2}{\beta^3}(a_{0i}v^i)(b_jv^j)+\frac{4}{\beta^2}(a_{0i}v^i)^2,
\end{equation*}
where $V=v^i(x)\frac{\partial}{\partial x^i}$ is a section of $\pi^*_A(TM)$.  

Taking now into account the Cauchy-Schwartz inequality
\begin{equation*}
\sqrt{a_{ij}y^iy^j}\sqrt{a_{ij}v^iv^j}\geq a_{ij}y^iv^j,
\end{equation*}
we obtain
\begin{equation*}
\begin{split}
g_{ij}(x,y)v^iv^j& \geq \frac{2}{\beta^2}X^2+\frac{3\alpha^4}{\beta^4}Y^2-\frac{8\alpha^2}{\beta^3}XY+\frac{4}{\beta^2}X^2\\
& =\frac{3\alpha^4}{\beta^4}\Bigl[(Y-\frac{4}{3}\frac{\beta}{\alpha^2}X)^2+\frac{2}{9}\frac{\beta^2}{\alpha^4}X^2   \Bigr]\geq 0,
\end{split}
\end{equation*}
where we denote for simplicity $X:=a_{0i}v^i$, $Y:=b_iv^i$. 

Obviously, the equality holds for 
\begin{equation*}
X=0, \qquad Y=0, 
\end{equation*}
i.e.
\begin{equation*}
a_{ji}v^iy^j=0,\qquad b_iv^i=0.
\end{equation*}

By taking the derivative with respect to $y^i$ of the first equality and taking into account that $a_{ij}$ is positive definite it follows $v^i=0$, i.e. $g_{ij}$ is positive definite on $A$. 

$\qedd$
\end{proof}


\subsection{Admissible curves and geodesics.}\label{Admissible curve}
 Let us consider a smooth curve $c:[a,b]\to M$ with velocity $\dot{c}$ in a conic Finsler space $(M,F)$ with the conic Finsler metric $F:A\to (0,\infty)$. Since the expression $F(c, \dot{c})$ does not always make sense, one needs restrict to curves where it does. Following \cite{JS} we define:

\begin{definition}(\cite{JS})
Let $F:A \longrightarrow (0, \infty)$ be a conic Finsler metric on $M$ and    $a=t_0< \cdots<t_k=b$ be a partition of $[a, b]$. 
A piecewise smooth curve $c : [a, b] \longrightarrow M$ is called {\it $F$-admissible} if
\begin{enumerate}
\item for  $t \in (t_{i-1}, t_i )$,  $(i=1, \cdots, k)$, the derivative $\dot{c}(t)$  belongs to $A$,     
\item the right derivative $\dot{c}(t_{i-1}^+)$ and the left derivative $\dot{c}(t_i^-)$ belong to $A_{c(t_{i-1})}$  and $A_{c(t_i)}$ $(i=1, \cdots, k)$, respectively.
\end{enumerate}
 The $F$-{\it length} of an $F$-admissible curve $c$ is defined by
\begin{equation*}\label{length}
    \mathcal{L}_F(c)=\sum_{i=1}^{k}\int_{t_{i-1}}^{t_i} F(\dot{c}(t))d t.  
\end{equation*}
\end{definition}

\begin{definition}(\cite{JS})
Let $F:A \longrightarrow (0, \infty)$ be a conic Finsler metric on $M$ and  $c : [a, b] \longrightarrow M$ be a piecewise smooth $F$-admissible curve   between two fixed points $p=c(a)$, $q=c(b)\in M$. Let  $a=t_0< \cdots<t_k=b$ be a partition of $[a, b]$. 
A {\it geodesic} of $F$ that joins points $p$ and $q$ is a critical curve of constant speed of the length  functional
\begin{equation*}\label{energy functional}
E_F(c)=\sum_{i=1}^k\int_{t_{i-1}}^{t_i} F(\dot{c}(t))d t.
\end{equation*}
\end{definition}

A detailed study of the geodesics of a conic Finsler metrics of Kropina type is going to be done elsewhere. 

Pursuing our attempt toward a description of U-Kropina metrics as solutions of Zermelo's navigation problem, we point out that under the influence of the unit length wind $W$, the time minimizing travel path is no longer a Riemannian $h$-geodesic, but a geodesic of the U-Kropina metric $F$ in the sense described above.


\subsection{Constant flag curvature.}

We are going to give in the following the description of Kropina spaces $(M, F=\alpha^2/\beta)$ of constant flag curvature $K$ by means of Zermelo's navigation problem. As mentioned above, a Kropina metric is a conic Finsler metric, and therefore in the constant flag curvature case it can be characterized as in Proposition \ref{Proposition 3.1}, where $K$ is constant.

Restricting to U-Kropina metrics, 
the necessary and sufficient conditions for a Kropina space to be of constant flag curvature $K$  can be easily obtained. Taking into account that $(M,F)$ is a conic Finsler metric defined on the conic domain $A_x$ for each $x\in M$, by restricting the computations in \cite{YO1} and \cite{YO2} to $A_x$
 we have
 
\begin{theorem}{\rm (\cite{YO1}, \cite{YO2})}\label{Theorem 3.2}

Let $(M, F=\alpha^2/\beta)$ be an $n(\ge 2)$-dimensional U-Kropina space and let  $h_{ij}(x)$,  $W=W^i(\partial/\partial x^i)$ be given as in \eqref{2.5}, provided \eqref{2.5'}.

Then, the Kropina space $(M, F)$ is  of constant flag curvature $K$ if and only if the following conditions hold:
 
\begin{enumerate}
\item $W=W^i(\partial/\partial x^i)$ is a Killing vector field on $(M,h)$, that is $W_{i||j}+W_{j||i}=0$, where $"_{||}"$ represents the covariant derivative with respect to $h$.
\item The Riemannian space $(M, h)$ is of constant sectional curvature $K$.
\end{enumerate}
\end{theorem}


\section{On the existence of Kropina spaces.}\label{section 4}

In this section, we shall discuss the existence of globally defined Kropina spaces  on  an $n(\ge 2)$-dimensional differential manifold $M$ and give some examples.

Let us remark that the characterization of the Kropina space $(M,F=\frac{\alpha^2}{\beta})$ in terms
 of the Riemannian metric $h$ and the vector field $W$ used in  Theorem \ref{Theorem 3.2} requires the existence of a unit length vector field globally defined on $(M,h)$. This requirement imposes already topological restrictions on the manifold $M$. 
 
 Besides the notion of U-Kropina metric introduced already, the following definition is natural.

\begin{definition}\label{definition 5.1}
Let $(M, h)$ be an $n(\ge 2)$-dimensional Riemannian space and $W$ be a vector field globally defined on $M$. We denote by $F$ the Kropina metric constructed by means of these $h$ and $W$ as explained in Section \ref{section 1}. 

If $W$ is a unit length Killing vector field on $(M,h)$, then $F$ will be called {\it Kropina metric with unit Killing vector field}, or {\it UK-Kropina metric}.

Moreover, if  $W$ is a unit length Killing vector field on the Riemannian space $(M,h)$ of constant sectional curvature, then $F$ will be called {\it constant flag curvature Kropina metric}, or {\it CC-Kropina metric}.
\end{definition}

Obviously, the set of  U-Kropina metrics includes the set of UK-Kropina metrics. Similarly, the set of UK-Kropina metrics includes the set of CC-Kropina metrics.

We start by studying the existence of Kropina metrics that are represented by a Riemannian metric $h$ and a unit Killing vector field $W$ on $M$. These are the allowable spaces for 
UK-Kropina spaces and furthermore for 
CC-Kropina spaces.


\subsection{Globally defined U-Kropina metrics.}\label{subsection 4.1}

By definition it follows that 
any Riemannian manifold $(M,h)$ that admits a globally defined nowhere vanishing vector field $V$ can be endowed with a globally defined U-Kropina metric.

In order to see this, remark that for a Riemannian metric $h$ and a vector field $V$ on $M$ without zeros, it is enough to normalize $V$, i.e. we define
\begin{equation*}
W:=\frac{1}{\sqrt{h(V,V)}}V.
\end{equation*}
One can construct now a U-Kropina metric using $h$ and $W$ as described in Section \ref{section 1}.

The study of differential manifolds that admits nowhere vanishing vector fields is an old and basic topic in differential topology and we do not enter here in details. However, among known results we mention some that are connected with the examples given in this paper.

\begin{proposition}
Let $(M,h)$ be a Riemannian manifold that satisfies one of the following conditions
\begin{enumerate}
\item $M$ is connected, non compact,
\item $M$ is an odd dimensional sphere,
\item $M$ is compact and orientable with zero Euler-Poincar\'e characteristic.
\end{enumerate}
Then $M$ admits globally defined U-Kropina metrics.
\end{proposition}

\begin{remark}
One can directly construct concrete examples of U-Kropina metrics on such manifolds.
\end{remark}

\subsection{Globally defined UK-Kropina metrics.}\label{subsection 4.2}

In the following we will construct globally defined  UK-Kropina metrics on  Riemannian manifolds that admit a unit Killing vector field. 

Let us recall two important results on Killing vector fields of unit length that will be useful later on. 

\begin{lemma}{\rm (\cite{BN2})}\label{Lemma 4.4}
Let $W$ be a Killing vector field on a Riemannian space $(M, h)$. Then, the following conditions are equivalent:
\begin{enumerate}
\item $W$ has constant length on $M$,
\item $W$ is parallel with respect to the Levi-Civita connection of $h$,
\item every integral curve of the vector field $W$ is a geodesic in $(M, h)$.
\end{enumerate}
\end{lemma}

\begin{lemma}{\rm (\cite{BN2})}\label{Lemma 4.5}
Each two-dimensional Riemannian manifold $M$ with a unit Killing vector field $W$ is locally Euclidean. Thus $M$ is isometric to the Euclidean plane or to one of the flat complete surfaces: the cylinder, the torus, the M\"obius band and the Klein bottle. 

Moreover, the vector field $W$ is parallel. 
\end{lemma}

\begin{remark}
In the last two cases (i.e. the M\"obius band and the Klein bottle) the unit Killing vector field $W$ is 
quasiregular, i.e. there exits integral curves of $W$ of different length) and has the unique singular circle trajectory. One can see that $W$ is defined up to multiplication by -1. 

In all other cases, the vector field $W$ is regular, i.e. all its integral curves are closed and have one and the same length, and $W$ may have any direction.
\end{remark}

Without trying to list up all Riemannian manifolds admitting a nowhere zero  Killing vector field, we mention the most important cases.

Firstly, in {\bf low dimension}, we have

\begin{theorem} \label{Theorem 4.17}
If $M$ is one of the following
\begin{enumerate}
\item a locally Euclidean Riemannian surface, or
\item a compact connected simply connected three-dimensional Riemannian space,
\end{enumerate}
then $M$ admits a unit length Killing vector field, and hence a globally defined UK-Kropina metric. 
\end{theorem}
\begin{proof}
\begin{enumerate}
\item It follows from Lemma \ref{Lemma 4.5}.
\item The proof is based on the famous  {\it Poincar\'e conjecture} recently proved by G. Perelman (see \cite{KL} for a detailed review as well as \cite{BN1}, p. 6), namely that any arbitrary compact connected simply connected metrizable topological $3$-dimensional manifold $M$ with the second countability axiom is homeomorphic to the $3$-sphere $\Sph^3$.

Using this and taking into account  the geometry of $\Sph^3$ with the canonical Riemannian metric of sectional curvature 1, one can easily see that it admits a Killing vector field of unit length.

\end{enumerate}

$\qedd$
\end{proof}


\begin{example}(Euclidean plane)\label{UK-Kropina on E^2}

On the {\it Euclidean plane} $\E^2=\{(x^{1}, x^{2})|x^{1}, x^{2} \in \R\}$ endowed with the canonical Euclidean metric $h_{ij}=\delta_{ij}$,
 a  unit Killing vector field can be obtained by the parallel translation by a  given unit vector $v=(h^{1}, h^{2})$, 
where $(h^{1})^2+(h^{2})^2=1$. We define the $1$-parameter isometry group $\{\phi_t\}$ by
 \[	\phi_t : (x^{1}, x^{2})\longrightarrow (x^{1}+h^{1}t, x^{2}+h^{2}t)  \]
for every $t\in \R$.
The vector field $W$ defined by 
\[	W_{(x^{1}, x^{2})}:=\frac{d}{dt}\phi_t(x^{1}, x^{2})\bigg|_{t=0}=(h^{1}, h^{2})  \]
is Killing and of constant unit length. Actually this vector field is a parallel one.

Since $W_{(x^1, x^2)}=h^1\frac{\partial}{\partial x^1}+h^2\frac{\partial}{\partial x^1}$, the conic domain 
$A_{(x^1, x^2)}$ 
of the UK-Kropina metric 
induced by the navigation data $(h,W)$ is
\[A_{(x^1, x^2)}=\{(y^1, y^2)\in T_{(x^1, x^2)}\E^2|h^1y^1+h^2y^2>0\}.\]
\end{example}


\begin{example}(The cylinder)\label{cylinder}
 
On the $cylinder$ $S^1\times \R$, where $\R$ is the set of real numbers, endowed with the metric induced from the canonical Euclidean metric of $\E^3$, the 
 $\Sph^1$-action $\eta : \Sph^1\times (\Sph^1\times {\R})\longrightarrow \Sph^1\times {\R}$ is defined by
 \[	\eta (e^{ti}, (w, k))=(e^{ti}w, k),\hspace{0.2in} e^{ti}\in \Sph^1,\hspace{0.2in} (w, k)\in \Sph^1\times \R.  \]
 Denoting $(e^{ti}w, k)$  by  $\eta_t(w, k)$, it results that $\{ \eta_t \}$ is a $1$-parameter isometry group.
Hence, the vector field $W$ is defined by
\[	W_{(w, k)}:=\frac{d}{dt}\eta_{t}(w, k)\bigg|_{t=0}
	                  =(iw, 0)  \]
is Killing.
Since the metric $h$ on the cylinder $\Sph^1\times \R$ is given by
\[	h((w, u), (w, u))=|w|^2+u^2,  \]
we have
\[	h(W_{(w, k)}, W_{(w, k)})=h((iw, 0), (iw, 0))  =|iw|^2   =1.      \]
Therefore, $W$ is a unit Killing vector field.
 Furthermore, since all integral curves of a constant length Killing vector field on a Riemannian manifold are geodesics (see Lemma \ref{Lemma 4.4})
 , it follows that the closed curve $\eta_t(w, k)=(e^{ti}w, k)$ is a geodesic passing through the point $(w, k)$ on the cylinder  $S^1\times \R$.

In order to describe the conic domain of the UK-Kropina metric $F=\alpha^{2}/\beta$ induced by $(h,W)$ we make use of the real coordinates. Indeed, by putting $w:=\cos\theta+i\sin\theta\in\Sph^1$ it results
\begin{equation*}
W_{(w,k)}=(-\sin\theta,\cos\theta,0)=\frac{\partial}{\partial \theta},
\end{equation*}
and therefore in the basis $(\dfrac{\partial}{\partial\theta},\dfrac{\partial}{\partial k})$ of 
$T_{(w,k)}(\Sph^1\times\R)$ we have $W^1=1$, $W^2=0$. The conic domain reads
\begin{equation*}
A_{(w,k)}=\{(y^1,y^2)\in T_{(w,k)}(\Sph^1\times\R)|y^1>0\}.
\end{equation*}

 \end{example}
 

\begin{example}(The torus)\label{torus}

On the $torus$ $T^1=\Sph^1\times \Sph^1$ endowed with the metric induced by the canonical Hermitian metric on $\C^2$,  the $\Sph^1$-action $\eta : \Sph^1\times (\Sph^1\times \Sph^1)\longrightarrow \Sph^1\times \Sph^1$ is defined by
 \[	\eta (e^{it}, (w, \tilde{w}))=(e^{it}w, e^{ti}\tilde{w}),
	\hspace{0.2in}e^{it}\in \Sph^1, \hspace{0.2in}(w, \tilde{w})\in \Sph^1\times \Sph^1.  \]
 Denoting $(e^{it}w, e^{it}\tilde{w})$  by   $\eta_t(w, \tilde{w})$, it results again that $\{ \eta_t \}$ is a $1$-parameter isometry group, hence the vector field $X$ defined by
\[	X_{(w, \tilde{w})}:=\frac{d}{dt} (e^{it}w, e^{it}\tilde{w})\bigg|_{t=0}    =(iw, i\tilde{w})  \]
is Killing.    Since the metric $h$ on the torus $\Sph^1\times \Sph^1$ is given by
\[	h((w, \tilde{w}), (w, \tilde{w}))=|w|^2+|\tilde{w}|^2,  \]
we have
\[	h(X_{(w, \tilde{w})}, X_{(w, \tilde{w})})=h((iw, i\tilde{w}), (iw, i\tilde{w}))    =|iw|^2+|i\tilde{w}|^2    =2.   \]
Thus, the new vector field $W$  defined by
\[	W:=\frac{1}{\sqrt{2}}X  \]
is a unit length Killing vector field.
By the same reason as above it follows that the curve $(e^{it}w, e^{it}\tilde{w})$ is a closed geodesic passing through the point $(w, \tilde{w})$.

By putting 
\begin{equation*}
w:=\cos\theta^1+i\sin\theta^1,\qquad
\tilde w:=\cos\theta^2+i\sin\theta^2
\end{equation*}
it follows
\begin{equation*}
W_{(w, \tilde{w})}=\frac{1}{\sqrt{2}}(-\sin\theta^1,\cos\theta^1,-\sin\theta^2,\cos\theta^2)
=\frac{1}{\sqrt{2}}(\frac{\partial}{\partial \theta^1}+\frac{\partial}{\partial \theta^2})
\end{equation*}
and therefore in the basis $(\frac{\partial}{\partial \theta^1},\frac{\partial}{\partial \theta^2})$ we have $W^1=W^2=\frac{1}{\sqrt{2}}$. The conic domain of the UK-Kropina metric induced by $(h,W)$ is
\begin{equation*}
A_{(w, \tilde{w})}=\{(y^1,y^2)\in T_{(w, \tilde{w})}(\Sph^1\times\Sph^1)|y^1+y^2>0\}.
\end{equation*}


\end{example}


In {\bf higher dimension}, Lie groups are manifolds that admit unit length Killing vector fields. 
\begin{theorem}
If $M$ is one of the following
\begin{enumerate}
\item a compact Lie group,
\item a Lie group with a bi-invariant metric,
\end{enumerate}
then $M$ can be endowed with a globally defined UK-Kropina metric. 
\end{theorem}

\begin{proof}
\begin{enumerate}
\item
Let $G$ be a compact Lie group, and let $L_h:G\longrightarrow G$ and $R_h\longrightarrow G$ denote left and right  multiplication, respectively. 

It is known that a compact Lie group $G$ admits a bi-invariant metric $g$ defined by
\[	g_p:=L_{p^{-1}}^*g_e, \]
where $g_e$ is a conjugation invariant metric on $T_e G$, $p\in G$ and $e$ is the identity (see for eg. \cite{V}). 

Choose a vector $v\in T_eG$ such that $g_e(v, v)=1$. Then, there exist a left-invariant vector field $W$ such that $W_e=v$.
If $\phi_t$ is the local flow of $W$ through $e$, then 
\[	W_h=\frac{\partial R_{\phi_t}(h)}{\partial t}\bigg|_{t=0}. \]

On the other hand, suppose that $Y$ and $Z$ are left-invariant vector fields.
Since $g$ is a bi-invariant metric, we have
\[	g(Y, Z)=g(R_{\phi_t*}L_{\phi_{-t}*}Y, R_{\phi_t*}L_{\phi_{-t}*}Z).   \]
Using the left invariance of $Y$ and $Z$, we get
\[	g(Y, Z)=g(R_{\phi_t*}Y, R_{\phi_t*}Z).  \]
This means that $R_{\phi_t*}$ is an isometry of $G$. Hence, it follows that the left invariant vector field $W$ is a 
Killing vector field and satisfies the condition $g_h(W_h, W_h)=g_e(v, v)=1$. 

\item If $M$ is a Lie group that admits a bi-invariant metric, then the same proof as in the compact case applies. 
\end{enumerate}

$\qedd$

\end{proof}

\begin{remark}
These are enough for our purpose, but many other spaces (for example symmetric and homogeneous Riemannian spaces) that admit globally defined UK-Kropina metrics can be obtained from \cite{BN1}, \cite{BN2}.
\end{remark}

On the other hand, taking into account the basis properties of Killing vector fields of constant length, we obtain the following rigidity result.

\begin{theorem}\label{Theorem 4.18}
If $M$ is one of the following
\begin{enumerate}
\item a compact even-dimensional Riemannian space with positive sectional curvature, 
\item a Riemannian space with negative Ricci tensor,
\item a two-dimensional Riemannian space whose sectional curvature at a point is negative,
\item an $n(\ge 3)$-dimensional Riemannian space of constant negative curvature,
\end{enumerate}
then $M$ does not admit a globally defined unit Killing field, nor globally defined UK-Kropina metrics. 
\end{theorem}

\begin{proof}
\begin{enumerate}
\item It follows from the following {\it Theorem of Berger's} (\cite{Ber}):
every Killing vector field on a compact even-dimensional Riemannian space $(M, g)$ with positive sectional curvature vanishes at some point in $M$.
\item 
If $X$ is a  Killing vector field of unit length on an $n$-dimensional Riemannian space $(M, h)$, 
then the Ricci tensor $Ric$ of the space $(M, h)$ must satisfy the condition $Ric(X, X)\ge 0$.
Moreover, the equality $Ric(X, X)\equiv 0$ is equivalent to the parallelism of the vector field $X$ (see \cite{BN1} for details) and the statement follows.
\item It follows from the following property (see \cite{BN1} for details):
If $h$ and $X$ are as above, and denoting by $K$ and $R$ the sectional curvature and the curvature tensor of $h$, respectively, then 
 for any point $x\in T_M$, we have 
 \begin{equation}\label{curv_lemma}
	K(w, X_x)=g(R(w, X_x)X_x, w)\ge 0,   
	\end{equation}
where $w\in T_xM$ is unit length and it satisfies the condition $w\perp X_p$.
\item It follows immediately from \eqref{curv_lemma}.
\end{enumerate}
$\qedd$

\end{proof}


\begin{example}(The three dimensional sphere $\Sph^3$)

We consider $\Sph^3$ as a subset of $\E^4$ with the naturally induced metric, namely 
\[\Sph^3:=\{(x^1, x^2,x^3,x^4)\in \E^4|(x^1)^2+(x^2)^2+(x^3)^2+(x^4)^2=1\} \hookrightarrow \E^{4}
\]
with the parametrization 
\[x^1=\cos{u^3}\cos{u^1}, \hspace{0.1in}  x^2=\cos{u^3}\sin{u^1}, \hspace{0.1in}  x^3=\sin{u^3}\cos{u^2}, \hspace{0.1in}  x^4=\sin{u^3}\sin{u^2},\]
where $0\le u^1<2\pi$, $0\le u^2<2\pi$ and $0\le u^3\leq\pi/2$.   

The Riemannian metric $h$ on $\Sph^{3}$ reads
\[ds^2=\cos^2{u^3}(du^1)^2+\sin^2{u^3}(du^2)^2+(du^3)^2.\]

For $t\in [0,2\pi)$, we define the action  $\varphi_t: \Sph^3\longrightarrow \Sph^3$ 
as follows:
\begin{equation*}
\begin{split}
\varphi_t:&(\cos{u^3}\cos{u^1}, \cos{u^3}\sin{u^1}, \sin{u^3}\cos{u^2},\sin{u^3}\sin{u^2}) \\
 & \longmapsto (\cos{u^3}\cos{(u^1+t)}, \cos{u^3}\sin{(u^1+t)}, \sin{u^3}\cos{(u^2+t)},\sin{u^3}\sin{(u^2+t)}),
\end{split}
\end{equation*}
and put
\begin{eqnarray*}
  W_x:&=&\frac{\partial \varphi_t}{\partial t}\bigg|_{t=0}\\
            &=&-\cos{u^3}\sin{u^1}\frac{\partial}{\partial x^1}+\cos{u^3}\cos{u^1}\frac{\partial}{\partial x^2}
                                              -\sin{u^3}\sin{u^2}\frac{\partial}{\partial x^3}+\sin{u^3}\cos{u^2}\frac{\partial}{\partial x^4}\\
           &=&\frac{\partial}{\partial u^1}+\frac{\partial}{\partial u^2},
\end{eqnarray*}   
where $x=(x^1, x^2,x^3,x^4)\in \E^{4}$.
Hence, if we define
\[W_{x}:=W^1\frac{\partial}{\partial u^1}+W^2\frac{\partial}{\partial u^2}+W^3\frac{\partial}{\partial u^3},\]
we have $W^1=W^2=1$ and $W^3=0$.

It follows $|W|=\sqrt{h_{11}(W^1)^2+h_{22}(W^2)^2+h_{33}(W^3)^2}=\sqrt{\cos^2{u^3}+\sin^2{u^3}}=1$, i.e. $W$ is a unit vector field on $\Sph^{3}$.

A straightforward computation shows that $W$ is in fact a unit length Killing vector field on $\Sph^{3}$.
Moreover, the conic domain of the UK-Kropina metric $F$ induced by the navigation data $(h,W)$ is  
$A_x=\{(v^1, v^2, v^3)\in T_x\Sph^3|v^1+v^2>0\}$.

\end{example}


\section{Kropina spaces of constant curvature.}\label{section 5}

From
Theorem \ref{Theorem 3.2} it follows that CC-Kropina structures can be constructed only on Riemannian spaces of constant sectional curvature.

\begin{remark}
Let $(M, h)$ be one of such Riemannian spaces. Then $(M, h)$ is locally isometric to a Riemannian space form (see Corollary 2.2 in Chapter 8 of \cite{Ca}). Namely, for any point $p\in M$ there exists a coordinate neighborhood $U$ of $p$ which is isometric to a coordinate neighborhood $\tilde{U}$ of a space form. Denoting the isometry between $U$ and $\tilde{U}$ by $\phi$, from Lemma 7.3 which is shown later, $\phi$ lifts to a conic isometry between the Kropina metric on $U$ constructed by means of the unit Killing vector field $W$
and the Kropina metric on $\tilde{U}$ constructed by means of $\phi_*(W)$.
Hence, there is no harm in considering the Kropina spaces on Riemannian space forms.
\end{remark}

It is well known that, up to local Riemannian isometry, the universal covering of a Riemannian space 
$(M,h)$ of constant sectional curvature $K$
 is one of the 
model spaces
 \begin{enumerate}
\item $\mathbb H^n$, if $K=-1$,
\item $\E^n$, if $K=0$,
\item $\Sph^n$, if $K=1$.
\end{enumerate}

\begin{remark}
It is known that in order to obtain a Riemannian metric of constant sectional curvature $K$ it is enough to multiply one of the above models' metrics by $\frac{1}{K}$.
\end{remark}

Remark next that from Theorem \ref{Theorem 4.18}  it follows that $\mathbb H^n$ and $\Sph^{2m}$ does not admit a Killing vector field of constant length.
Hence, globally defined UK-Kropina metrics cannot be constructed on them.

On the other hand, we can easily construct Killing vector fields of constant length on $\E^n$ and $\Sph^{2m-1}$.

We have studied the case of $\E^2$ in Example \ref{UK-Kropina on E^2}. It can be extended to arbitrary dimension as follows.
 Choose any unit vector $\textbf{a}=(a^1, \cdots, a^n)\in \Sph^{n-1}$.
For any element ${\bf x}=(x^1, \cdots, x^n)\in \E^n$, define a mapping $\varphi_t : \E^n \longrightarrow \E^n$ by 
\[	\varphi_t : {\bf x}=(x^1, \cdots, x^n) \longmapsto {\bf x}+t\textbf{a}=(x^1+t a^1, \cdots, x^n+t a^n). \]

Since $\{\varphi_t\}$ is a 1-parameter isometry group of $\E^n$, the vector field $W$ defined by
\[	  W_{{\bf x}}:=\frac{d}{dt}\varphi_t({\bf x}) \bigg|_{t=0}={\bf a}  \]
is a unit length Killing vector field.

Next, we consider the case of $\Sph^{2m-1}$ regarded as the subset of $\mathbb{C}^m$ given by
\[	\Sph^{2m-1}=\{(z^1, \cdots, z^m)| z^i\in \mathbb{C} , 1\le i \le m, \ |z^1|^2+\cdots +|z^m|^2=1\}.\]
Putting $s=e^{it}\in \Sph^1$ and defining an $\Sph^1$-action $\varphi_t$ on $\Sph^{2m-1}$ by
\[	\varphi_t : (z^1, \cdots, z^m) \longmapsto (sz^1, \cdots, sz^m)=(e^{it}z^1, \cdots, e^{it}z^m), \]
it follows that $\{\varphi_t\}$ is a 1-parameter isometry group of $\Sph^{2m-1}$.
The vector field $W$ defined by
\[	W_z:=\frac{d}{dt}\varphi_t(z)  \bigg|_{t=0}=(iz^1, \cdots, iz^m),  \]
where $z=(z^1, \cdots, z^m)$, is a   Killing vector field of unit length because of 
\[	|W_z|^2=|iz^1|^2+ \cdots +|iz^m|^2=|z^1|^2+ \cdots +|z^m|^2=1.\]

Therefore, we get

\begin{theorem}\label{Theorem 4.20}
The only manifolds (up to Riemannian local isometry) that admits CC-Kropina structures are the Euclidean space $\E^n$, $n\geq 2$  and odd dimensional spheres $S^{2m-1}$, $m\geq 2$.
\end{theorem}


Inspired by \cite{BRS} we will attempt a classification of CC-Kropina structures. We start by recalling the following important result:

\begin{lemma}{\rm (\cite{BRS})}\label{Basic Lemma}
Let $P_i=P_i(x)$ be solutions of the following system:
\begin{equation*}\label{equation 1.1}
	\frac{\partial P_i}{\partial x^j}+\frac{\partial P_j}{\partial x^i}=0.  
\end{equation*}
Then
\[	P_i=Q_{ij}x^j+C_i,  \]
where $(C_i)$ is an arbitrary constant row vector and $Q=(Q_{ij})$ is an arbitrary constant skew-symmetric matrix, i.e.,  $Q_{ij}=-Q_{ji}$.
\end{lemma}


\subsection{The Euclidean case.}
The first Riemannian space form we consider is the standard Euclidean space.
The admissible vector fields $W$ for CC-Kropina structures are described in the following proposition.

\begin{proposition}\label{Proposition 6.3}
Let $(\E^n, F=\alpha^2/\beta)$ be the Kropina space induced by the navigation data made of the flat metric $h=(\delta_{ij})$  and a vector field $W=W^i(\partial/\partial x^i)$ of unit length  on $\E^n$, where $(x^i)$ are the standard coordinates of $\E^n$.
Then, $(\E^n, F)$ is of constant flag curvature $K=0$ if and only if $W$ has the form
\[	W^i(x)=C^i,  \]
where  $(C^i)$ is a  column vector of unit length.
\end{proposition}

\begin{proof}
Obviously $(\E^n, F)$ is of constant flag curvature $K=0$ if and only if $W=W^i(\partial/\partial x^i)$ is a Killing vector field of unit length on $(\E^n,h)$, that is
\begin{equation}\label{condition 1}
W_{i\vert\vert j}+W_{j\vert\vert i}=0,
\end{equation}
where the indices are lowered by $h_{ij}=\delta_{ij}$,  the symbol $_{||}$ represents the covariant derivative with respect to the Euclidean metric  $h$, and 
\begin{equation}\label{condition 2}
\sum_{i=1}^n(W_i)^2=1.
\end{equation}

First, we consider the condition (\ref{condition 1}).
Keeping in mind that in this case the covariant derivative  $_{||}$ is simply partial differentiation, we have
\[\frac{\partial W_i}{\partial x^j}+\frac{\partial W_j}{\partial x^i}=0.   \]
From Lemma \ref{Basic Lemma}, it follows that  $W$ is of the form
\begin{equation}\label{equation 1.2}
	W_i=Q_{ij}x^j+C_i,
\end{equation}
where $(C_i)$ is an arbitrary constant row vector and $Q=(Q_{ij})$ is an arbitrary constant skew-symmetric matrix, that is,  $Q_{ij}=-Q_{ji}$.
Therefore, $W$ is a Killing vector field if and only if $W$ is in the form (\ref{equation 1.2}).

Next, we consider the condition (\ref{condition 2}). Substituting (\ref{equation 1.2}) in (\ref{condition 2}), we have
\[	\sum_{i=1}^n({Q^i}_jx^j+C^i)^2=1,   \]
that is,
\[	\sum_{i=1}^n {Q^i}_r{Q^i}_sx^rx^s+2\sum_{i=1}^n {Q^i}_rx^rC^i  + \sum_{i=1}^n(C^i)^2=1,  \]
where ${Q^i}_j=\delta^{ir}Q_{rj}$.
Since the above equation must  hold good for any $(x^i)$, it follows that all ${Q^i}_r$ must vanish and $\sum_{i=1}^n(C^i)^2=1$.
Hence, the conclusion  follows.

$\qedd$
\end{proof}


\subsection{The spherical case.}\label{subsection spherical case}
\subsubsection{The projective coordinate system on a unit  $n$-sphere.}

Consider the $(n+1)$-dimensional Euclidean space $\E^{n+1}$ and an $n$-sphere
\[	\Sph^n:=\{(x^0, x^1, \cdots, x^n) | (x^0)^2+(x^1)^2+\cdots+(x^n)^2=1\}.  \]
We define the eastern and western hemispheres by
\[	\Sph^n_+=\{(x^0, x^1, \cdots, x^n) |   (x^0, x^1, \cdots, x^n)\in \Sph^n,\hspace{0.2in} x^0>0\}  \]
and 
\[	\Sph^n_-=\{(x^0, x^1, \cdots, x^n) |  (x^0, x^1, \cdots, x^n)\in \Sph^n,\hspace{0.2in} x^0<0\} , \]
respectively.
We call the points $p_+:=(1, 0, \cdots, 0)$ and $p_-:=(-1, 0, \cdots, 0)$ the eastern pole and the western pole, respectively.

We consider the tangent spaces $T\Sph^n_+ :=\{  (1, x^1, x^2, \cdots, x^n)\} \simeq \E^n$ and  $T\Sph^n_- :=\{  (-1, x^1, x^2, \cdots, x^n)\} \simeq \E^n$ at the eastern pole and   the western pole, respectively.     The projection 
\[ \varphi_{\pm} :    T\Sph^n_{\pm}\simeq \E^n    \longrightarrow \Sph^n_{\pm}  \]
is defined by  
\[\varphi _{\pm}: \textbf{x}      \longmapsto          \bigg(        \frac{\pm 1}{\sqrt{1+\textbf{x}\cdot \textbf{x}}},    \hspace{0.1in}  \frac{1}{\sqrt{1+\textbf{x}\cdot \textbf{x}}}\textbf{x}    \bigg),  \]	 
where the notation $"\cdot"$ stands for the standard inner product on $\E^n$.

For any curve $\textbf{x}(t)$  on $T\Sph^n_{\pm}\simeq \E^n$, we get the curve 
\[c_{\pm}(t):=\bigg(  \frac{\pm1}{\sqrt{1+\textbf{x}(t)\cdot \textbf{x}(t)}},          \hspace{0.1in}   \frac{1}{\sqrt{1+\textbf{x}(t)\cdot \textbf{x}(t)}}\textbf{x}(t)          \bigg)  \]
on $\Sph^n_{\pm}$,
and therefore
\begin{eqnarray}\label{equation 1.3}
     &&c'_{\pm}(t)\\
       &:=&\bigg(  \frac{ \mp \textbf{x}(t)\cdot \textbf{x}'(t)}{(\sqrt{1+\textbf{x}(t)\cdot \textbf{x}(t)})^3 },           \hspace{0,2in} 
                           -\frac{\textbf{x}(t)\cdot \textbf{x}'(t)}{(\sqrt{1+\textbf{x}(t)\cdot \textbf{x}(t)})^3 }\textbf{x}(t)
	                                     +\frac{1}{\sqrt{1+\textbf{x}(t)\cdot \textbf{x}(t)}}\textbf{x}'(t)
	         	           \bigg)  \nonumber
\end{eqnarray}

Denoting the metric on $\Sph^n$ by $h$ and putting $\textbf{x}'(t)=\textbf{y}(t)$, we get
\begin{eqnarray*}
	&&h(c'_\pm(t), c'_\pm(t))\\
	&=&  \frac{(\textbf{x}(t)\cdot \textbf{y}(t))^2}{(1+\textbf{x}(t)\cdot \textbf{x}(t))^3 }
                    + \frac{(\textbf{x}(t)\cdot \textbf{y}(t))^2      (\textbf{x}(t)\cdot \textbf{x}(t))}{(1+\textbf{x}(t)\cdot \textbf{x}(t))^3 }
	            -\frac{2(\textbf{x}(t)\cdot \textbf{y}(t))^2}{(1+\textbf{x}(t)\cdot \textbf{x}(t))^2}
	            +  \frac{\textbf{y}(t)\cdot \textbf{y}(t)}{1+\textbf{x}(t)\cdot \textbf{x}(t)}              \\
           &=& \frac{\textbf{y}(t)\cdot \textbf{y}(t)}{1+\textbf{x}(t)\cdot \textbf{x}(t)}
             -\frac{(\textbf{x}(t)\cdot \textbf{y}(t))^2}{(1+\textbf{x}(t)\cdot \textbf{x}(t))^2},
\end{eqnarray*}
that is,
\[	h(c'_\pm(t), c'_\pm(t))=\frac{(\textbf{y}(t)\cdot \textbf{y}(t))(1+\textbf{x}(t)\cdot \textbf{x}(t)) 
	                -(\textbf{x}(t)\cdot \textbf{y}(t))^2  }{(1+\textbf{x}(t)\cdot \textbf{x}(t))^{2}}  \]
Hence, the length of the tangent vector $\textbf{y}\in T_{\textbf{x}}\E^n$ with respect to the metric $h$ is given by
\begin{equation*}
	h(\textbf{y}, \textbf{y})=\frac{(\textbf{y}\cdot \textbf{y})(1+\textbf{x}\cdot \textbf{x}) 
	                -(\textbf{x}\cdot \textbf{y})^2  }{(1+\textbf{x}\cdot \textbf{x})^{2}}.
\end{equation*}


\subsubsection{The $(2m-1)$-sphere. }\label{sec.6.2.2}

Let $(M, h)$ be a $(2m-1)$-sphere endowed with the standard Riemannian metric of constant sectional curvature $K(>0)$. This is the only sphere admitting a Killing vector field of unit length, as already shown. Multiplying the standard Riemannian metric on $\Sph^{2m-1}$ 
by $1/{K}$  we have
\[	h(\textbf{y}, \textbf{y})=\frac{1}{K}\frac{(\textbf{y}\cdot \textbf{y})(1+\textbf{x}\cdot \textbf{x}) 
	                -(\textbf{x}\cdot \textbf{y})^2  }{(1+\textbf{x}\cdot \textbf{x})^{2}},  \]
or, equivalently, by putting $\textbf{x}=(x^1, \cdots, x^{2m-1})$ and $\textbf{y}=(y^1, \cdots, y^{2m-1})$, we get
\[	h(\textbf{y}, \textbf{y})=\frac{1}{K}\frac{(1+\textbf{x}\cdot \textbf{x})\sum_{i=1}^{2m-1}(y^i)^2-\sum_{i,j=1}^{2m-1}x_ix_jy^iy^j}
	                {(1+\textbf{x}\cdot \textbf{x})^{2}},   \]
where $x_i=\delta_{ij}x^j$, $i,j=1,\dots,2m-1$ and thus
\begin{equation}\label{equation 1.5}
	h_{ij}=\frac{1}{K}\bigg(\frac{\delta_{ij}}{(1+\textbf{x}\cdot \textbf{x})}-\frac{x_ix_j}{(1+\textbf{x}\cdot \textbf{x})^2}\bigg),\quad i,j=1,\dots,2m-1.
\end{equation}
From the  equation (\ref{equation 1.5}),  we have
\begin{equation}\label{equation 1.6}
	h^{ij}=K(1+\textbf{x}\cdot \textbf{x})(\delta ^{ij}+x^ix^j),
\end{equation}
and from here
\begin{eqnarray*}
	\frac{\partial h_{ij}}{\partial x^k}
	=\frac{1}{K}\bigg\{-\frac{2x_k\delta_{ij}}{(1+\textbf{x}\cdot \textbf{x})^2}
	       -\frac{\delta_{ik}x_j+\delta_{jk}x_i}{(1+\textbf{x}\cdot \textbf{x})^2}
	  +\frac{4x_ix_jx_k}{(1+\textbf{x}\cdot \textbf{x})^3}\bigg\}
\end{eqnarray*}
and 
\begin{eqnarray*}
	^h\gamma_{jrk}&=&\frac{1}{2}\bigg\{\frac{\partial h_{kr}}{\partial x^j}
	                +\frac{\partial h_{jr}}{\partial x^k}
                       -\frac{\partial h_{jk}}{\partial x^r}\bigg\}\\
           &=&\frac{1}{2K}\bigg\{
               -\frac{2x_j\delta_{kr}}{(1+\textbf{x}\cdot \textbf{x})^2}
	          -\frac{2x_k\delta_{jr}}{(1+\textbf{x}\cdot \textbf{x})^2} 
	         +\frac{4x_kx_rx_j}{(1+\textbf{x}\cdot \textbf{x})^3}  \bigg\}\\
           &=&-\frac{1}{K(1+\textbf{x}\cdot \textbf{x})^2}\bigg(
               x_j\delta_{kr}	          +x_k\delta_{jr}
	         -\frac{2x_kx_rx_j}{1+\textbf{x}\cdot \textbf{x}}  \bigg).
\end{eqnarray*}
Using (\ref{equation 1.6}), we get
\begin{eqnarray*}
	^h{{\gamma_j}^i}_k&=&-\frac{1}{1+\textbf{x}\cdot \textbf{x}}(\delta ^{ir}+x^ix^r)
	                                \bigg(   x_j\delta_{kr}	          +x_k\delta_{jr}
	                            -\frac{2x_kx_rx_j}{1+\textbf{x}\cdot \textbf{x}}  \bigg) \\
	&=&-\frac{1}{1+\textbf{x}\cdot \textbf{x}}
	         \bigg(x_j {\delta^i}_k   +x_k{\delta^i}_j \bigg).	
\end{eqnarray*}

Therefore, we have
\begin{equation*}
W_{i\vert\vert j}=\frac{\partial W_i}{\partial x^j}-{W_r}^h{{\gamma_i}^r}_j
	         =\frac{\partial W_i}{\partial x^j}+\frac{1}{1+\textbf{x}\cdot \textbf{x}}
                                (x_i W_j   +x_jW_i ).
\end{equation*}

The equation of the Killing vector field  can be now written as follows:
\begin{equation}\label{equation 1.7}
	\frac{\partial W_i}{\partial x^j}+\frac{\partial W_j}{\partial x^i}	+\frac{2}{1+\textbf{x}\cdot \textbf{x}}  (x_i W_j   +x_jW_i )=0. 
\end{equation}
Putting $P_i:=K(1+\textbf{x}\cdot \textbf{x})W_i$, we have
\[	\frac{\partial P_i}{\partial x^j}
=2Kx_jW_i+K(1+\textbf{x}\cdot \textbf{x})\frac{\partial W_i}{\partial x^j}.  \]
Then, the equation (\ref{equation 1.7}) is rewritten as
\[	\frac{\partial P_i}{\partial x^j}+\frac{\partial P_j}{\partial x^i}=0,  \]
and from Lemma \ref{Basic Lemma}, we get
\begin{eqnarray*}
	P_i=Q_{ij}x^j+C_i,
\end{eqnarray*}
where $(C_i)$ is an arbitrary constant row vector and $Q=(Q_{ij})$ is an arbitrary constant skew-symmetric matrix.
Therefore, we obtain
\begin{eqnarray*}
	W_i=\frac{Q_{ij}x^j+C_i}{K(1+\textbf{x}\cdot \textbf{x})}.
\end{eqnarray*}
From (\ref{equation 1.6}), we get
\begin{equation*}
W^i=(\delta ^{ij}+x^ix^j)(Q_{jr}x^r+C_j)
	  	  ={Q^i}_rx^r+C^i+(\textbf{x} \cdot \textbf{C})x^i,
\end{equation*}
where ${Q^i}_j:=\delta^{is}Q_{sj}$ and $C^i:=\delta^{is}C_s$.

Using the condition $|W|^2:=h_{ij}W^iW^j=1$ we compute
\begin{equation*}
\begin{split}
|W|^2:&=\frac{1}{K}\bigg(\frac{\delta_{ij}}{(1+\textbf{x}\cdot \textbf{x})}-\frac{x_ix_j}{(1+\textbf{x}\cdot \textbf{x})^2}\bigg)\\
	   &    \qquad \qquad \qquad \bigg({Q^i}_rx^r+C^i+(\textbf{x} \cdot \textbf{C})x^i\bigg)
             \bigg({Q^j}_sx^s+C^j+(\textbf{x} \cdot \textbf{C})x^j\bigg)\\
     &=\frac{1}{K(1+\textbf{x}\cdot \textbf{x})}\bigg(Q_{jr}x^r+C_j\bigg) \bigg({Q^j}_sx^s+C^j+(\textbf{x} \cdot \textbf{C})x^j\bigg)\\
            &=\frac{1}{K(1+\textbf{x}\cdot \textbf{x})}
         \bigg(Q_{jr}x^r{Q^j}_sx^s+2Q_{jr}x^rC^j
                           + (\textbf{C} \cdot \textbf{C}) +(\textbf{x} \cdot \textbf{C})^2 \bigg),
\end{split}
\end{equation*}
and thus we have
\begin{equation}\label{equation 1.8}
	Q_{jr}x^r{Q^j}_sx^s+2Q_{jr}x^rC^j
                           + (\textbf{C} \cdot \textbf{C}) +(\textbf{x} \cdot \textbf{C})^2
    =K(1+\textbf{x}\cdot \textbf{x}).
\end{equation}

Hence, it follows that the equation (\ref{equation 1.8}) holds for any $\textbf{x}\in \E^{2m-1}$ if and only if the equations
\begin{equation}\label{equation 1.9}
	Q_{jr}{Q^j}_s+C_rC_s=K\delta_{rs}, \hspace{0.2in}Q_{jr}C^j=0,  \hspace{0.2in}\textbf{C} \cdot \textbf{C}=K.
\end{equation}
hold. Therefore we get

\begin{proposition}
Suppose that  a metric on a $(2m-1)$-sphere $\Sph^{2m-1}$ of constant sectional curvature $K>0$ is given by multiplying a standard Riemannian metric by $1/K$. 
Denote the eastern hemisphere and the western hemispheres of $\Sph^{2m-1}$ by  $\Sph^{2m-1}_+$ and $\Sph^{2m-1}_-$, and the tangent spaces at the eastern pole and the western pole by $T\Sph^{2m-1}_+$ and $T\Sph^{2m-1}_-$, respectively. 

Then, using the projective coordinate system  $(x^i)$ on $\Sph^{2m-1}_+$ (or $\Sph^{2m-1}_-$) which is mapped from $T\Sph^{2m-1}_+$ (or $T\Sph^{2m-1}_-$),  a Killing vector filed $W$ of unit  length on $\Sph^{2m-1}_+$ (or $\Sph^{2m-1}_-$) can be written as
\[W^i={Q^i}_rx^r+C^i+(\textbf{x} \cdot \textbf{C})x^i,\]
where ${Q^i}_r$ and $C^i$ satisfy the conditions (\ref{equation 1.9}).
\end{proposition}


\subsection{Classification of Kropina spaces with constant curvature.}\label{Classification of CC-K}
\subsubsection{The classification theorem.}
\begin{theorem}\label{Theorem 1}
Let $(M,F=\alpha^2/\beta$)
 be a Kropina space on a smooth manifold $M$ of dimension $n\ge 2$, where $\alpha=\sqrt{a_{ij}(x)y^iy^j}$ and $\beta=b_i(x)y^i$.
Then $(M, F)$ is of constant flag curvature $K$ if and only if the following conditions are satisfied.
\begin{enumerate}
\item The Riemannian manifold $(M,h)$ is a Riemannian space of constant sectional curvature and the vector field $W=W^{i}(x)\dfrac{\partial}{\partial x^{i}}$ is a unit length Killing vector field with respect to $h$, where the Riemannian metric $h=(h_{ij})$ and the functions $W^{i}(x)$ are given in \eqref{2.5}, provided \eqref{2.5'}.
\item 
Up to local isometry, the constant sectional curvature Riemannian metric $h$ and the vector field $W$ must belong to one of the following two families.

$(+)$ When $K>0$ : $h$ is $\frac{1}{K}$ times the standard metric on the unit $(2m-1)$-sphere $\Sph^{2m-1}$ in projective coordinates, and $W=W^i(\partial/\partial x^i)$ is a unit length Killing vector field, where
\[	W^i={Q^i}_rx^r+C^i+(\textbf{x} \cdot \textbf{C})x^i,  \]
with
\[	Q_{jr}{Q^j}_s+C_rC_s=K\delta_{rs}, \hspace{0.2in}Q_{jr}C^j=0,  \hspace{0.2in}\textbf{C} \cdot \textbf{C}=K.  \]
In these coordinates, the quadratic form of $h$, evaluated on $y\in T_x\Sph^{2m-1}$, satisfies
\[	h(\textbf{y}, \textbf{y})=\frac{1}{K}\bigg\{\frac{(\textbf{y}\cdot \textbf{y})(1+\textbf{x}\cdot \textbf{x}) 
	                -(\textbf{x}\cdot \textbf{y})^2  }{(1+\textbf{x}\cdot \textbf{x})^2}\bigg\}. \]

$(0)$  When $K=0$ : $h$ is the Euclidean metric $\delta_{ij}$ on $\E^n$ and $W=W^i(\partial/\partial x^i)$ 
is a unit length Killing vector field, where
\begin{equation*}
	W^i=C^i,
\end{equation*}
that is, $\sum_{i=1}(C^i)^2=1$.
\end{enumerate}
\end{theorem}


\subsubsection{Globally defined Killing vector fields on the standard sphere $\Sph^{2m-1}$.} 
Let $W$ be a Killing vector field on the tangent space $T\Sph^{2m-1}_+ \simeq \E^{2m-1}$ at the eastern pole. Projecting $W$ to the eastern hemisphere by $\varphi_{+*}$ and using the equation (\ref{equation 1.3}), we get
\begin{eqnarray*}
     \varphi_{+ *}(W)
       &=&\begin{pmatrix} -\frac{  \textbf{x}\cdot W}{(\sqrt{1+\textbf{x}\cdot \textbf{x}})^3 }  \cr
                           -\frac{\textbf{x}\cdot W}{(\sqrt{1+\textbf{x}\cdot \textbf{x}})^3 }\textbf{x}
	                                     +\frac{1}{\sqrt{1+\textbf{x}\cdot \textbf{x}}}W        	 \end{pmatrix}  
        =\begin{pmatrix}  -\frac{  \{\textbf{x}\cdot Q \textbf{x}+  \textbf{x}  \cdot C +(\textbf{x}  \cdot C )(\textbf{x} \cdot \textbf{x})\}}{(\sqrt{1+\textbf{x}\cdot \textbf{x}})^3 } \cr
                    -\frac{ \textbf{x}\cdot Q \textbf{x}  }{(\sqrt{1+\textbf{x}\cdot \textbf{x}})^3 }\textbf{x}
	                                     +\frac{Q  {\bf x}+C     }{\sqrt{1+\textbf{x}\cdot \textbf{x}}}        \end{pmatrix}  
          =\begin{pmatrix}  -\frac{   \textbf{x}  \cdot C }{\sqrt{1+\textbf{x}\cdot \textbf{x}} }   \cr
                  	                                     \frac{Q  {\bf x}+C     }{\sqrt{1+\textbf{x}\cdot \textbf{x}}}  \end{pmatrix}  \nonumber\\
        &=&\begin{pmatrix} -\frac{   (\sum_{i=1}^n x^i C^i) }{\sqrt{1+\textbf{x}\cdot \textbf{x}} }  \cr
                  	                                     \frac{({Q^i}_jx^j+C^i)     }{\sqrt{1+\textbf{x}\cdot \textbf{x}}}    \end{pmatrix}  
       = \begin{pmatrix}   0 &  - (C^i)^t  \cr
                                            (C^i) & ({Q^i}_j)\end{pmatrix}.
                      \begin{pmatrix}       \frac{ 1}{\sqrt{1+\textbf{x}\cdot \textbf{x}}}\cr
                                     \frac{(x^j)}{\sqrt{1+\textbf{x}\cdot \textbf{x}}}  \end{pmatrix},
                          \nonumber\\
\end{eqnarray*}
where ${\bf x}$ is a column vector.
In the above equation, we have used the relation ${\bf x}\cdot Q{\bf x}=0$. So, we get
\begin{equation}\label{equation 1.12}
          [\varphi_{+ *}(W)]^t=p^t\Omega,
\end{equation}
where

\begin{equation}\label{equation 1.13}
     p=\begin{pmatrix} \frac{ 1}{\sqrt{1+\textbf{x}\cdot \textbf{x}}}\cr  \frac{1}{\sqrt{1+\textbf{x}\cdot \textbf{x}}}\textbf{x}   \end{pmatrix}, \hspace{0.2in}
     \Omega=\begin{pmatrix}   0 &   C^t  \cr
                                           - C & -Q\end{pmatrix}.
\end{equation}

We must notice that $p$ in (\ref{equation 1.13}) is a point in the eastern hemisphere. Since we have 
\[  \lim_{|{\bf x}| \longrightarrow \infty} \frac{ 1}{\sqrt{1+\textbf{x}\cdot \textbf{x}}} =0, \hspace{0.1in}
        \lim_{|{\bf x}| \longrightarrow \infty} \bigg|\frac{1}{\sqrt{1+\textbf{x}\cdot \textbf{x}}}\textbf{x} \bigg|=1,\]
for the sake of the continuity of $W$ on the whole of $\Sph^{2m-1}$ the value of the projection of $W$ at any point $p$ on the equator is defined by 
$p^t\Omega$.
Furthermore, we extend $W$ to the open western hemisphere by imposing
  \[[\varphi_{- *}(W)]^t=p^t\Omega, \hspace{0.1in}  p=\begin{pmatrix}        -\frac{ 1}{\sqrt{1+\textbf{x}\cdot \textbf{x}}} \cr \frac{1}{\sqrt{1+\textbf{x}\cdot \textbf{x}}}\textbf{x} \end{pmatrix}\]
It follows
\[W=Q{\bf x}-C-({\bf x}\cdot C){\bf x}.\]

Therefore, there exists a skew-symmetric matrix $\Omega$ in  (\ref{equation 1.12})  and   (\ref{equation 1.13}) for a Killing vector field on $\Sph^{2m-1}$. 


\subsubsection{Globally defined Killing vector fields on $\E^{n}$.} 

Since the Euclidean space $\E^{n}$ is covered by a single coordinate chart, the vector $W^{i}=C^{i}$ in 
Proposition \ref{Proposition 6.3} is a globally defined unit length Killing vector field. 

\begin{remark}
For $n=2$ the flat Riemannian metric on cylinder, torus, M\"obius band and Klein bottle together with the corresponding unit Killing vector fields 
provide examples of CC-Kropina structures on these surfaces, respectively (see Examples \ref{cylinder}, \ref{torus} for concrete constructions on cylinder and torus).
\end{remark}


\section{The moduli space $\mathcal{M}_\mathcal{K}$.}\label{section 6}
\subsection{The isometry between two conic Kropina metrics.}
\begin{definition}
Let  $(M_1, F_1)$ and $(M_2, F_2)$  be two conic Finsler spaces, where $F_i : A_i \longrightarrow (0, \infty)$ $(i=1,2)$ are conic Finsler metrics.
We say that 
$(M_1, F_1)$ and $(M_2, F_2)$   are  {\it conic isometric} if there exists a diffeomorphism $\phi : M_1 \longrightarrow M_2$ which, when lifted to a map $\widetilde \phi :A_1\subset TM_{1}\to A_2\subset TM_{2}$, satisfies $\widetilde\phi^*F_2=F_1$ (in fact $\widetilde\phi$ is just the differential map of $\phi$).
\end{definition}

\begin{remark}
We point out that in the definition above we assume $\widetilde \phi(A_{1})=\phi_{*} (A_1)= A_2$. This is always implicitly assumed when we discuss conic isometries.
\end{remark}

Let $M_1$, $M_2$ be $n(\ge 2)$-dimensional differential manifolds. Consider two Kropina metrics $F_i$ on $M_i$, where $F_i=(\alpha_i)^2/\beta_i$, $\alpha_{i}=\sqrt{(a_{i})_{jk}y^{j}y^{k}}$ and 
$\beta_{i}=(b_{i})_{j}y^{j}$, $i=1,2$. They are constructed by the pairs $(h_i, W_i)$ of a Riemannian metric $h_i$ and a unit vector field $W_i$ on $M_i$, for $i=1,2$, respectively. We get

\begin{lemma}\label{Lemma 2.1}
Let  $(M_1, F_1)$ and $(M_2, F_2)$  be two conic Finsler spaces, where $F_i : A_i \longrightarrow (0, \infty)$ $(i=1,2)$ are conic Finsler metrics.
Let $\phi : M_1 \longrightarrow M_2$ be a diffeomorphism. The following three statements are equivalent:\\
$(i)$        $\phi$ lift to a conic  isometry between $F_1$ and $F_2$.\\
$(ii)$        $\phi^*\alpha_2=e^{\frac{\tau(x)}{2}}\alpha_1$ and $\phi^*\beta_2=e^{\tau(x)}\beta_1$, where $\tau(x)=\log\frac{\phi^{*}(b_{2})^{2}}{(b_{1})^{2}}$ is a function  of position alone.\\
$(iii)$        $\phi^*h_2=h_1$ and $\phi_*W_1=W_2$.
\end{lemma}
								
\begin{proof}
(i)$\Rightarrow$(ii)

Assume (i) and remark that the isometry condition $\widetilde\phi^*F_2=F_1$ reads
\begin{equation}\label{7.1}
(\widetilde\phi^{*}\alpha_{2})^{2}\cdot\beta_{1}=(\alpha_{1})^{2}\cdot\widetilde\phi^{*}\beta_{2}.
\end{equation}

We point out that we regard the two Riemannian metrics $\alpha_{1}$, $\alpha_{2}$ as well as the linear 1-forms $\beta_{1}$, $\beta_{2}$ as mappings $TM_{i}\to \R$, $i=1,2$, respectively.

Moreover, we see that $(\alpha_{1})^{2}$, $(\widetilde\phi^{*}\alpha_{2})^{2}$ and $\beta_{1}$, $\widetilde\phi^{*}\beta_{2}$ are homogeneous polynomials of degree 2 and 1 in $y$, respectively. Since 
$(\alpha_{1})^{2}$ is not divisible by $\beta_{1}$ (otherwise the rank of the matrix $(a_{1})_{ij}$ would decrease and this is not allowed by definition for a Riemannian metric), it follows that $\widetilde\phi^{*}\beta_{2}$ must be divisible by $\beta_{1}$, i.e. we must have $\widetilde\phi^{*}\beta_{2}=f\cdot\beta_{1}$, for a function $f:M_{1}\to \R$. Substituting this in \eqref{7.1} it results 
$(\widetilde\phi^{*}\alpha_{2})^{2}=f\cdot(\alpha_{1})^{2}$ and therefore we must have $f>0$. By denoting $\tau(x):=\log f(x)$ the first two relations in (ii) are obtained. 

From here we have
\begin{equation}\label{b2}
\widetilde\phi^{*}(b_{2})^{2}=e^{\tau(x)}\cdot(b_{1})^{2}
\end{equation}
taking into account that 
\begin{equation*}
\widetilde\phi^{*}(b_{2})_{i}=e^{\tau(x)}\cdot(b_{1})_{i},\qquad 
\widetilde\phi^{*}(a_{2})_{ij}=e^{\tau(x)}\cdot(a_{1})_{ij},
\end{equation*}

Therefore the expression for $\tau(x)$ follows immediately.

The converse (ii)$\Rightarrow$(i) follows immediately.

\bigskip

(ii)$\Rightarrow$(iii)

Recall that the geometric data $(\alpha_{i},\beta_{i})$ are related to Zermelo's navigation data $(h_{i},W_{i})$ by relations \eqref{2.5} making use of the functions $k_{i}(x)$, for $i=1,2$.

If we assume (ii), then from \eqref{b2} it follows
\begin{equation*}
\tau=k_{1}-\phi^{*}k_{2}.
\end{equation*}

Relation \eqref{2.5} for $h_{ij}$ implies 
\begin{equation*}
\begin{split}
\phi^{*}(h_{2})_{ij}&=\phi^{*}(e^{k_{2}}\cdot (a_{2})_{ij})=\phi^{*}(e^{k_{2}})\cdot \phi^{*}((a_{2})_{ij})
=e^{\phi^{*}k_{2}}e^{\tau}\cdot(a_{1})_{ij}\\
&=e^{\tau-k_{	1}+\phi^{*}k_{2}}\cdot (h_{1})_{ij}=(h_{1})_{ij}
\end{split}
\end{equation*}
and similarly for $W$
\begin{equation*}
\phi^{*}(W_{2})_{i}=\frac{1}{2}e^{\phi^{*}k_{2}}\phi(b_{2})_{i}=(W_{1})_{i}.
\end{equation*}

The conclusion follows immediately taking into account that, for $W_{i}\in TM_{i}$ as vector fields, $\phi_{*}W_{1}=W_{2}$ is equivalent to $W_{1}=\phi^{*}(W_{2})$, where 
$W_{1}\equiv(W_{1})_{i}(x)y^{i}$, and similarly for $W_{2}$, are regarded as mappings on $TM_{i}$, $i=1,2$, respectively. 

$\qedd$
\end{proof}

\begin{remark}
It is interesting to remark that a Kropina isometry $\phi:(M_{1},F_{1})\to (M_{2},F_{2})$ do not induce a Riemannian isometry between Riemannian spaces $(M_{1},\alpha_{1})$ and $(M_{2},\alpha_{2})$, but it induces one between spaces $(M_{1},h_{1})$ and $(M_{2},h_{2})$.
\end{remark}


\subsection{A Lie group formalism.}
Theorem \ref{Theorem 1} characterizes 
the CC-Kropina spaces of constant flag curvature $K$ in terms of pairs $(M,h)$, where $h$ is a Riemannian metric of constant sectional curvature and $W$ a unit length Killing vector field on $(M,h)$. 
 
 Indeed, for any CC-Kropina space $(M,F)$ with the corresponding Zermelo's navigation data $(h, W)$, there exists a Riemannian local isometry $\phi$ between $(M, h)$ and  one of the two standard Riemannian space forms: the positive constant sectional curvature sphere $(\Sph^{2m-1}, h_+)$, and the flat Euclidean space $(\E^n, h_0)$.

From Lemma \ref{Lemma 2.1} it follows that $\phi$ lifts to a local Finslerian conic isometry between $(M, F)$ and the CC-Kropina spaces obtained from the Zermelo's navigation data sets $(h_{+},\phi_{*}W)$ and $(h_{0},\phi_{*}W)$ on $\Sph^{2m-1}$ and $\E^n$, respectively, where $W$ is the corresponding unit length Killing vector field listed up in Theorem \ref{Theorem 1}, and $\phi$ is an isometry of $\E^n$ and $\Sph^{2m-1}$, respectively.

This correspondence includes a certain degree of redundancy due to the presence of isometric Kropina structures. If we denote by $Isom_{h}(M)$ the isometry group of the Riemannian structure $(M,h)$, where $h$ is one of the canonical metrics $h_{+}$ or $h_{0}$, then we would like to evaluate the redundancy of the CC-Kropina metrics given by the pairs $(h,\phi_{*}W)$, where $\phi\in Isom_{h}(M)$ (see \cite{BRS} for a similar study in the case of Randers spaces). In other words we identify all isometric CC-Kropina metrics with a point and compute the dimension of the moduli space of CC-Kropina metrics obtained in this way.  

The Lie group theory apparatus used in order to do this is similar to the setting in  \cite{BRS}. Indeed, 
begin with a standard Riemannian space form $(M,h)$, 
where $M=\E^n\slash \Sph^{2m-1}$, 
$h=h_+\slash h_0$, respectively, and identify the isometry group $G=Isom_{h}(M)$ with a matrix subgroup of $GL_{n+1}\R$.
The Killing vector field $W$ of $h$ gives a representation of a matrix Lie subalgebla $\mathfrak{h}$ of $\mathfrak{gl}_{n+1}\R$, the Lie algebra of  $GL_{n+1}\R$.
The push-forword action $W \longmapsto \phi_*W:=\phi_*\circ W \circ \phi^{-1}$ on the manifold corresponds to the "adjoint action"
\[ \Omega \longmapsto Ad_g\Omega:=g\Omega g^{-1}\]
of $G$ on $\mathfrak{h}$, where
$g\in GL_{n+1}R$ is the matrix which  corresponds to the isometry map $\phi$, and $\Omega \in \mathfrak{h}$ is the matrix analog of the Killing vector field $W$. Obviously,
$Ad : \mathfrak{h} \longrightarrow \mathfrak{h}$ is well defined because the equation $\mathcal{L}_Wh=0$ becomes $\mathcal{L}_{\phi_*W}h=0$ under the action of the isometry map $\phi$.
Thus, $\phi_*W$ is an Killing vector field whenever $W$ is one.

The adjoint action $Ad$ described above partitions $\mathfrak{h}$ into orbits. Each orbit corresponds to a distinct conic isometry class of Kropina metrics with constant flag curvature $K$. For each orbit, matrix theory singles out a privileged representative $\tilde{\Omega}$, i.e.  a normal form.

 One can see that for $K>0$, the metric $h=h_+$ is $\frac{1}{K}$ times the standard metric on the unit $\Sph^{2m-1}$. The orbits are those which result from the adjoint action of the orthogonal group $O(2m)$ on its Lie algebra $\mathfrak{o}(2m)$.
 
On the other hand, for $K=0$, we have $h=h_0$, the standard flat metric on $\E^n$. The orbits come from the adjoint action of the Euclidean group $E(n)$ on its Lie algebra $\mathfrak{E}(n)$ which is made of  $O(n)$ and the additive group $\R^n$ of translations.


We recall the following proposition: 

\begin{proposition} (\cite{BRS})\label{compact case}
Let $\Omega$ be any real $l\times l$ skew-symmetric matrix. 
Then, there exists an orthogonal matrix $B\in O(l)$ such that $\tilde{\Omega}=B^{-1}\Omega B$, where the matrix $\tilde{\Omega}$ defined as follows:
\begin{description}
\item when $l$ is even,
\[ \tilde{\Omega}:=a_1J\oplus \cdots \oplus a_mJ, \hspace{0.2in}    m=\frac{l}{2},  \] 
\item      when $l$ is odd,
\[   \tilde{\Omega}:=a_1J\oplus \cdots \oplus a_mJ \oplus 0, \hspace{0.2in}    m=\frac{l-1}{2},   \] 
\end{description}
where 
\[a_1\ge a_2\ge \cdots \ge a_m\ge 0, \hspace{0.1in} and \hspace{0.1in} J=\begin{pmatrix}0 & 1\cr -1 &0\end{pmatrix}.\]
\end{proposition}


\subsection{The $(2m-1)$-sphere.}
It is known that $Isom_{h_{+}}(\Sph^{2m-1})=O(2m)$, i.e. orthogonal matrices which represent rigid rotations by right multiplying the row vectors of $\E^{2m}$.
Each Killing vector field $W$ of $(\Sph^{2m-1}, h_{+})$  corresponds to a constant  skew-symmetric $2m\times 2m$ matrix  
\[\Omega:=\begin{pmatrix}0 & C^t \cr -C & -Q\end{pmatrix}\in \mathfrak{o}(2m).\]
This correspondence between Killing vector fields of $(\Sph^{2m-1}, h)$ and elements of $\mathfrak{o}(2m)$ is a Lie algebra isomorphism.

Applying Proposition \ref{compact case}, for $l=2m$,  we see that there exists a $g\in \mathfrak{o}(2m)$ such that 
\begin{equation}\label{Omega}
 g \Omega g^{-1}=\tilde{\Omega}=a_1J\oplus \cdots \oplus a_mJ. 
\end{equation}

The matrix $\tilde{\Omega}$ represents the Killing vector field $\tilde{W}=\phi_*W$, where $\phi$ is the map which corresponds to the orthogonal matrix $g$.
According to Theorem \ref{Theorem 1}, $\tilde{W}$ has the  form $\tilde{Q}{\bf x}+\tilde{C} +({\bf x}\cdot \tilde{C}) {\bf x}$ with respect to the projective coordinates ${\bf x}$ which parametrize the eastern hemisphere.
Comparing the matrix 
\[\begin{pmatrix}0 & \tilde{C}^t \cr - \tilde{C} & -\tilde{Q}\end{pmatrix}\]
of $\tilde{W}$ with $\tilde{\Omega}$ given in (\ref{Omega}), we conclude that  
\begin{equation}\label{equation 2.1}
\tilde{C}^{t}=(a_1, 0, \cdots, 0)
\end{equation}
and
\begin{equation*}
-\tilde{Q}=0\oplus a_2J \oplus \cdots \oplus a_mJ.
\end{equation*}

From the equation (\ref{equation 1.9}), we have
\begin{eqnarray}\label{equation }
	\tilde{Q}_{jr}{\tilde{Q}^j}_s+\tilde{C}_r\tilde{C}_s&=&K\delta_{rs},     \label{equation 7.3}     \\
         \tilde{Q}_{jr}\tilde{C}^j&=&0,     \label{equation 7.4}\\
        \tilde{C} \cdot \tilde{C}&=&K  \label{equation 7.5}
\end{eqnarray}
and therefore  we get $\tilde{C}_1=a_1$ and $\tilde{C}_2=\cdots =\tilde{C}_{2m-1}=0$ from (\ref{equation 2.1}).
Substituting these in (\ref{equation 7.5}), we get
   $a_1= \sqrt{K}$.
Remark the equation (\ref{equation 7.4}) identically holds.
We have $\tilde{C}_1\tilde{C}_1=K$
and  $\tilde{Q}_{j (2r-2)}{\tilde{Q}^j}_{(2r-2)}=\tilde{Q}_{j (2r-1)}{\tilde{Q}^j}_{(2r-1)}=(a_r)^2$ $(r=2, \cdots, m)$.
Substituting the above equations in (\ref{equation 7.3}), we get
\[ (a_2)^2=(a_3)^2=\cdots =K,\]
that is, 
\begin{equation*}
 a_2=a_3=\cdots =a_m=\sqrt{K}.
\end{equation*}

Therefore, it follows
\begin{theorem}
The moduli space $\mathcal{M}_{\mathcal{K}}$ for $(2m-1)$-dimensional   Kropina metrics of constant flag curvature $K(> 0)$ consists of a single  point 
\[ (a_1, a_2, \cdots,  a_m)=(\sqrt{K}, \sqrt{K}, \cdots \sqrt{K}) \in \E^m,\]
i.e. there is only one pair $(h_{+},W)$, up to Riemannian isometry, that induces all CC-Kropina structures on $\Sph^{2m-1}$.
\end{theorem}


\subsection{The Euclidean space.}

It is known that $E(n):=Isom_{h_{0}}(\E^{n})$  consists of rotations, reflections, and translations.
 
From Proposition \ref{Proposition 1.2} it follows that the unit Killing vector fields on $\E^{n}$ can be written as $W=C^i(\partial/\partial x^i)$, where $(C^i)$ is a unit vector.

If  $(C_1^i)$ and $(C_2^i)$ are two such unit vectors, then there exists a matrix $g\in SO(n-1)$ such that $(C_1^i)^t=(C_2^i)^t g$.
Hence,  $G$-orbit of Killing vector fields on $\E^n$ is a single point. 
Therefore, we get

\begin{theorem}
The moduli space $\mathcal{M}_{\mathcal{K}}$ for $\E^n$   Kropina metrics of constant flag curvature $K=0$ consists of a single  point 
\[ (a_1, a_2, \cdots,  a_n)=(1, 0, \cdots 0) \in \E^n.\]
\end{theorem}

\section{Projectively flat U-Kropina metrics.}\label{section 7}

\subsection{Projectively flatness conditions for U-Kropina metrics.}
Let $(M, F)$ and $(M, \overline{F})$ be two classical Finsler spaces.  The Finsler metric $F$ is said to be projective to $\overline{F}$ if any geodesic of $(M, F)$ coincides with a geodesic of $(M, \overline{F})$ as a set of points and vice versa. Furthermore, if $(M, \overline{F})$ is a locally Minkowski space, the Finsler space $(M, F)$ is said to be projectively flat.

For a conic Finsler metric
we define
\begin{definition}\label{Definition 7.2}
 A Finsler space is projectively flat if and only if the Finsler space is with rectilinear extremals. 
\end{definition} 
 
Therefore, $(M, F)$ is projectively flat if and only if it satisfies the equations
\[F_{x^ry^j}y^r-F_{x^j}=0 \hspace{0.1in}(Hamel's \hspace{0.1in} relation).\]
In this case, $G^i(x, y)=Py^i$, where $P(x, y)$ is given by $P=(1/2F)F_{x^i}y^i$, and $y^i=dx^i/ds \in A_x$.

Necessary and sufficient conditions for a Kropina metric to be projectively flat were obtained by M. Matsumoto in \cite{M1}. 

By putting
\begin{eqnarray*}
	\texttt{R}_{ij}:=\frac{W_{i||j}+W_{j||i}}{2}, \hspace{0.1in}
	  \texttt{S}_{ij}:=\frac{W_{i||j}-W_{j||i}}{2}, \hspace{0.1in}
	&&{\texttt{R}^i}_j:=h^{ir}\texttt{R}_{rj},\hspace{0.1in}
        {\texttt{S}^i}_j:=h^{ir}\texttt{S}_{rj}\\	
	\texttt{R}_{i}:=W^r\texttt{R}_{ri},    \hspace{0.1in}
	 \texttt{S}_{i}:=W^r\texttt{S}_{ri},    \hspace{0.1in}
	&&\texttt{R}^{i}:=h^{ir}\texttt{R}_{r}, \hspace{0.1in}
       \texttt{S}^{i}:=h^{ir}\texttt{S}_{r},
\end{eqnarray*}
the main result in  \cite{M1} reads in terms of $(h,W)$:

\begin{lemma}{\rm (\cite{M1})}\label{lemma 8.1}
Let $(M, F)$ be a Kropina space induced by the navigation data $(h,W)$.
Then $F$ is projectively flat if and only if $W_i$ satisfies   $\texttt{S}_{ij}   =W_i\texttt{S}_j  -W_j\texttt{S}_i$ and the space is covered by coordinate neighborhoods in which there exist functions $\mu_i(x)$ satisfying
\begin{equation*}\label{equation 7.28}
 {{^h \gamma_j}^i}_k ={\delta^i}_j\mu_k  +{\delta^i}_k\mu_j      -\texttt{S}^i h_{jk}-W^i \texttt{R}_{jk}. 
\end{equation*}
 \end{lemma}

We can give now 
\begin{theorem}\label{proj flat charact}
Let $(M, F=\alpha^2/\beta)$ be a globally defined UK-Kropina space induced by the navigation data $(h,W)$. 

Then,  $F$ is projectively flat if and only if $W$ is parallel on the Riemannian space $(M, h)$ and $(M, h)$ is projectively flat.
\end{theorem}
\begin{proof}
We consider an $n(\ge2)$-dimensional  globally defined UK-Kropina space $(M, \alpha^2/\beta)$ and let $h$ and $W$ be a Riemannian metric and a unit Killing vector field defined on it by (\ref{2.5})  and (\ref{2.5'}), respectively. Then, we have $\texttt{R}_{ij}=0$ and $\texttt{S}_i=0$.
 
Suppose that $(M, \alpha^2/\beta)$ is projectively flat. From Lemma \ref{lemma 8.1}, we have $\texttt{S}_{ij}=0$ and $ {{^h \gamma_j}^i}_k ={\delta^i}_j\mu_k  +{\delta^i}_k\mu_j$.
From the first formula we have $W_{i||j}=0$, that is, $W$ is parallel. The second one means that the Riemannian space $(M, h)$ is projectively flat. 

Conversely, suppose that $W$ is parallel on $(M, h)$ and the Riemannian space $(M, h)$ is projectively flat. From the first assumption, we have $W_{i||j}=0$. So, we get $\texttt{R}_{ij}=\texttt{S}_{ij}=\texttt{S}_i=0$. Hence, the two conditions in Lemma \ref{lemma 8.1} hold good and therefore the  Kropina space $(M, \alpha^2/\beta)$ is projectively flat.

$\qedd$
\end{proof}

Using Beltrami's theorem we obtain
\begin{corollary}\label{cor 8.4}
Let $(M, \alpha^2/\beta)$ be an $n(\ge 2)$-dimensional Kropina space of constant flag curvature $K(\ge 0)$ induced by the navigation data $(M,h)$.

Then, the Kropina space $(M, \alpha^2/\beta)$   is projectively flat if and only if the vector field $W$ is parallel on $(M, h)$.
\end{corollary}

\begin{remark}
One can see that our Examples \ref{cylinder} and \ref{torus} are examples of projectively flat Kropina metrics of constant flag curvature $K=0$ on surfaces.
\end{remark}

More generally, we consider CC-Kropina metrics of arbitrary dimension.


\subsection{ The Euclidean case.} 
Since $h_{ij}=\delta_{ij}$ from $(0)$ of Theorem \ref{Theorem 1} it results that the covariant derivative on $(M, h)$, which is denoted by $(_{||})$,  is simply partial differentiation. Hence, the covariant derivative of $W_i(x)=\delta_{ir}W^r(x)=\delta_{ir}C^r$, where $(C^i)$ is a column vector of unit constant length, reads
\[W_{i||j}=\frac{\partial W_i}{\partial x^j}=0.\]

Therefore, from Corollary \ref{cor 8.4} it follows 

\begin{theorem}\label{Theorem 7.11}
The Kropina space $(\E^n, F)$  of constant flag curvature $K=0$ is projectively flat.
\end{theorem}


\subsection{The spherical case.}
From Subsection \ref{sec.6.2.2}, we have
\begin{eqnarray*}
	W_{i||j}    &=&\frac{\partial W_i}{\partial x^j}+\frac{1}{1+\textbf{x}\cdot \textbf{x}}  (x_i W_j   +x_jW_i ).
\end{eqnarray*}
and
\begin{eqnarray*}
	W_i=\frac{Q_{ir}x^r+C_i}{K(1+\textbf{x}\cdot \textbf{x})},
\end{eqnarray*}
where
\begin{eqnarray*}
	Q_{jr}{Q^j}_s+C_rC_s=K\delta_{rs}, \hspace{0.2in}Q_{jr}C^j=0,  \hspace{0.2in}\textbf{C} \cdot \textbf{C}=K.
\end{eqnarray*}
From the above equations, we get
\begin{equation*}
\begin{split}
W_{i||j} &=-\frac{2x_j(Q_{ir}x^r+C_i)}{K(1+\textbf{x}\cdot \textbf{x})^2}+\frac{Q_{ij}}{K(1+\textbf{x}\cdot \textbf{x})}
                             +\frac{1}{1+\textbf{x}\cdot \textbf{x}}  (x_i \frac{Q_{jr}x^r+C_j}{K(1+\textbf{x}\cdot \textbf{x})}  +x_j\frac{Q_{ir}x^r+C_i}{K(1+\textbf{x}\cdot \textbf{x})} ) \\
           &=\frac{1}{K(1+\textbf{x}\cdot \textbf{x})^2}\bigg((1+\textbf{x}\cdot \textbf{x})Q_{ij}+x_i(Q_{jr}x^r+C_j)-x_j(Q_{ir}x^r+C_i)\bigg).
\end{split}
\end{equation*}
If we suppose that the $W$ is parallel, then we get $(Q_{ij})=O$ and $(C_i)=(0)$. This contradict the condition $\textbf{C} \cdot \textbf{C}=K$. Therefore we obtain

\begin{theorem}\label{Theorem 7.12}
The Kropina space $(\Sph^{2m-1}, F)$ of constant flag curvature $K>0$ can not be projectively flat.
\end{theorem}

\begin{remark}
We can compare now the Beltrami's Theorem for Riemannian manifolds with our findings. CC-Kropina metrics which are not flat can not be projectively flat, i.e. Beltrami's Theorem do not extend to the Finslerian setting. 

On the other hand, any flat CC-Kropina metrics must be projectively flat, that is, Beltrami's theorem applies to the case of Kropina structures, but only in the flat case.

\end{remark}

   Taking all these into account, Theorem  \ref{proj flat charact} implies

\begin{corollary}
Let $(M, F=\alpha^2/\beta)$ be a globally defined UK-Kropina space induced by the navigation data $(h,W)$.

Then,  $(M, F)$ is projectively flat if and only if the Riemannian space $(M, h)$ is flat and 
$W$ is parallel with respect to $h$.
\end{corollary}


\bigskip

\medskip

\begin{center}
Ryozo Yoshikawa

\bigskip
 Hachiman Technical High School, \\
	 	 5 Nishinosho-cho Hachiman   \\
	 	 523-0816 Japan       
	 
\medskip
{\tt ryozo@e-omi.ne.jp}

\bigskip

Sorin V. SABAU\\

\bigskip

Department of Mathematics\\
Tokai University\\
Sapporo, 005\,--\,8601 Japan

\medskip
{\tt sorin@tspirit.tokai-u.jp}
\end{center}

\end{document}